\documentclass[12pt,a4paper]{amsart}%

\usepackage{amsfonts,amsmath,amssymb}
\usepackage{amssymb,amsthm,amsxtra}
\usepackage[usenames]{color}
\usepackage{amscd}
\usepackage{amsthm}
\usepackage{amsfonts}
\usepackage{amssymb}
\usepackage{amsmath}
\usepackage{graphicx}
\usepackage{hyperref}
\usepackage{enumerate}%
\usepackage[all]{xy}
\usepackage[usenames,dvipsnames]{xcolor}
\usepackage{mathrsfs}
\usepackage{comment}
\usepackage[capitalize]{cleveref}
\usepackage{cleveref}
 \newtheorem*{corollary*}{Corollary}
 \newtheorem*{construction*}{Construction}
 \newtheorem*{definition*}{Definition}
 \newtheorem*{notation*}{Notation}
 \newtheorem*{lemma*}{Lemma}
 \newtheorem*{theorem*}{Theorem}
 \newtheorem*{remark*}{Remark}
 \newtheorem*{example*}{Example}
 \newtheorem*{conjecture*}{Conjecture}
 \newtheorem*{condition*}{Condition}
 \newtheorem*{result*}{Result}
 \newtheorem*{property*}{Property}

 \newtheorem*{cor*}{Corollary}
 \newtheorem*{const*}{Construction}
 \newtheorem*{defn*}{Definition}
 \newtheorem*{notn*}{Notation}
 \newtheorem*{lem*}{Lemma}
 \newtheorem*{thm*}{Theorem}
 \newtheorem*{rem*}{Remark}
 \newtheorem*{exm*}{Example}
 \newtheorem*{conj*}{Conjecture}

 \newtheorem{lemma}{Lemma}[section]
 
 \newtheorem{remark}[lemma]{Remark}
 \newtheorem{example}[lemma]{Example}
 \newtheorem{theorem}[lemma]{Theorem}

 \newtheorem{thm}[lemma]{Theorem}
 \newtheorem{prop}[lemma]{Proposition}
 \newtheorem{lem}[lemma]{Lemma}
 \newtheorem{defn}[lemma]{Definition}
 
 \newtheorem{cor}[lemma]{Corollary}

 \newtheorem{rem}[lemma]{Remark}

 \newtheorem{introtheorem}{Theorem}

 \crefname{introtheorem}{theorem}{theorems}
 \Crefname{introtheorem}{Theorem}{Theorems}
  \newtheorem{introthm}[introtheorem]{Theorem}
   \crefname{introthm}{theorem}{theorems}
 \Crefname{introthm}{Theorem}{Theorems}

 \crefname{introdef}{definition}{definition}
 \Crefname{introdef}{Definition}{Definition}

  \crefname{introcorollary}{corollary}{corollaries}
 \Crefname{introcorollary}{Corollary}{Corollaries}

 \crefname{introquest}{Question}{Questions}
 \Crefname{introquest}{Question}{Questions}

 \newtheorem{introcor}[introtheorem]{Corollary}
   \crefname{introcor}{corollary}{corollaries}
 \Crefname{introcor}{Corollary}{Corollaries}
 
   \crefname{introconjecture}{conjectures}{conjectures}
 \Crefname{introconjecture}{Conjecture}{Conjectures}
 
    \crefname{introconj}{conjectures}{conjectures}
 \Crefname{introconj}{Conjecture}{Conjectures}

     \crefname{introlem}{lemma}{lemmas}
 \Crefname{introlem}{Lemma}{Lemmas}
 
 \crefname{introremark}{remark}{remarks}
 \Crefname{introremark}{Remark}{Remarks}
 
  \crefname{introrem}{remark}{remarks}
 \Crefname{introrem}{Remark}{Remarks}

 \newtheorem{introprop}[introtheorem]{Proposition}
   \crefname{introprop}{Proposition}{Propositions}
 \Crefname{introprop}{Proposition}{Propositions}

   \crefname{introdefn}{definition}{definitions}
 \Crefname{introdefn}{Definition}{Definitions}
 
   \crefname{intronotn}{notation}{notations}
 \Crefname{intronotn}{Notation}{Notations}
 
   \crefname{introtask}{task}{tasks}
 \Crefname{introtask}{Task}{Tasks}
 
  \crefname{introprob}{problem}{problems}
 \Crefname{introprob}{Problem}{Problems}
 
   \crefname{introquestion}{question}{questions}
 \Crefname{introquestion}{Question}{Questions}

   \crefname{introexm}{example}{example}
 \Crefname{introquestion}{Example}{Example}
 \crefname{theorem}{theorem}{theorems}
 \Crefname{theorem}{Theorem}{Theorems}
  \crefname{thm}{theorem}{theorems}
 \Crefname{thm}{Theorem}{Theorems}
  \crefname{corollary}{Corollary}{Corollaries}
 \Crefname{corollary}{Corollary}{Corollaries}
   \crefname{cor}{Corollary}{Corollaries}
 \Crefname{cor}{Corollary}{Corollaries}
   \crefname{conjecture}{conjectures}{conjectures}
 \Crefname{conjecture}{Conjecture}{Conjectures}

    \crefname{conj}{conjectures}{conjectures}
 \Crefname{conj}{Conjecture}{Conjectures}

     \crefname{lem}{lemma}{lemmas}
 \Crefname{lem}{Lemma}{Lemmas}

      \crefname{lemma}{Lemma}{Lemmas}
 \Crefname{lemma}{Lemma}{Lemmas}

 \crefname{remark}{remark}{remarks}
 \Crefname{remark}{Remark}{Remarks}

  \crefname{rem}{remark}{remarks}
 \Crefname{rem}{Remark}{Remarks}

   \crefname{rem}{remark}{remarks}
 \Crefname{rem}{Remark}{Remarks}

   \crefname{proposition}{Proposition}{Proposition}
 \Crefname{proposition}{Proposition}{Proposition}

    \crefname{prop}{Proposition}{Propositions}
 \Crefname{prop}{Proposition}{Propositions}

   \crefname{defn}{definition}{definitions}
 \Crefname{defn}{Definition}{Definitions}

   \crefname{notn}{notation}{notations}
 \Crefname{notn}{Notation}{Notations}

   \crefname{task}{task}{tasks}
 \Crefname{task}{Task}{Tasks}

  \crefname{prob}{problem}{problems}
 \Crefname{prob}{Problem}{Problems}

   \crefname{question}{question}{questions}
 \Crefname{question}{Question}{Questions}

\newcommand{\alp}{\alpha}

\newcommand{\gam}{\gamma}

\newcommand{\lam}{\lambda}

\newcommand{\eps}{\varepsilon}

\newcommand{\End}{\operatorname{End}}

\newcommand{\bC}{\mathbb{C}}

\newcommand{\A}{\mathbb{A}}

\newcommand{\C}{\mathbb{C}}

\newcommand{\bR}{\mathbb{R}}
\newcommand{\bZ}{\mathbb{Z}}

\newcommand{\R}{\mathbb{R}}

\newcommand{\bQ}{\mathbb{Q}}

\providecommand{\cP}{\mathcal{P}}
\providecommand{\cQ}{\mathcal{Q}}
\providecommand{\cR}{\mathcal{R}}

\providecommand{\sub}{\subset}

\newcommand{\enc}{E_n^{(c)}}
\newcommand{\fnc}{F_n^{(c)}}

\providecommand{\Z}{\mathbb{Z}}

\newcommand{\DimaN}[1]{{{#1}}}
\newcommand{\DimaM}[1]{{{#1}}}
\newcommand{\DimaL}[1]{{{#1}}}
\newcommand{\DimaK}[1]{{{#1}}}
\newcommand{\DimaJ}[1]{{{#1}}}
\newcommand{\DimaI}[1]{{{#1}}}
\newcommand{\DimaH}[1]{{{#1}}}
\newcommand{\DimaF}[1]{{{#1}}}
\newcommand{\DimaG}[1]{{{#1}}}
\newcommand{\DimaE}[1]{{{#1}}}
\newcommand{\DimaD}[1]{{{#1}}}

\newcommand{\DimaC}[1]{{{#1}}}
\newcommand{\DimaB}[1]{{{#1}}}

\newcommand{\Dima}[1]{{{#1}}}

\author{Joseph Bernstein}
\address{  School of Mathematical Sciences, Tel Aviv University,
 Ramat Aviv, Tel Aviv 69978, Israel.}
\email{bernstei@tauex.tau.ac.il}
\urladdr{http://www.math.tau.ac.il/~bernstei/}

\author{Dmitry Gourevitch}
\address{Faculty of Mathematics and Computer Science, Weizmann
Institute of Science, Herzl 234, Rehovot 7610001, Israel.}
\email{dimagur@weizmann.ac.il}
\urladdr{http://www.wisdom.weizmann.ac.il/~dimagur}

\author{Siddhartha Sahi}
\address{Department of Mathematics, Rutgers University, Hill Center -
Busch Campus, 110 Frelinghuysen Road Piscataway, NJ 08854-8019, USA}
\email{sahi@math.rugers.edu}
\urladdr{https://sites.math.rutgers.edu/~sahi/}

\date{\today}

\newcommand{\sms}{\smallskip}

\newcommand{\proofend}{\hfill$\Box$\smallskip}

\newcommand{\nc}{\newcommand}

\def\calP{{\mathcal P}}

\def\+{{\bf[P]}}
\def\*{(*)}

\def\tb{\textbf}

\def\1{{\bf 1.}}
\def\2{{\bf 2.}}
\def\3{{\bf 3.}}
\def\4{{\bf 4.}}
\def\5{{\bf 5.}}
\def\6{{\bf 6.}}
\def\7{{\bf 7.}}
\def\8{{\bf 8.}}
\def\9{{\bf 9.}}

\begin{document}


\title[On the classification of hypergeometric orthogonal polynomials on the real line]{On the classification of hypergeometric families of orthogonal polynomials on the real line}
\begin{abstract}

Several important families of orthogonal polynomials on the real line are called ``hypergeometric'' since they can be explicitly described in terms of some hypergeometric series  $_pF_q$ that uses the degree $n$ of the polynomial as a parameter. It is natural to ask if one can classify all such families. Indeed many classification results have been obtained in this direction, but only under the additional assumption that the polynomials are eigenfunctions of some second order operator. 

In this paper we initiate a new approach to this classification. We propose a definition of an \emph{HG family} that makes precise, but also generalizes, the notion of a ``hypergeometric'' family. Our main result is that there are exactly 10 types of orthogonal HG families, 8 from the well-known Askey scheme and 2 additional types of families that can be expressed in terms of Lommel polynomials.

Our methods in this paper are  algebraic.  In particular, we classify a wider class of \emph{quasi-orthogonal} HG families, and this classification is valid over an arbitrary field of characteristic zero.   We also define a more general  class of \emph{rational HG families} and prove a structure theorem for quasi-orthogonal families in this class. We  provide examples for such families, that are in particular new families of orthogonal polynomials of potential interest.
 
\end{abstract}
\maketitle

\section{Introduction}
The theory of orthogonal polynomials on the real line is a classical subject.
There are many classical families of orthogonal polynomials that are called \tb{Hypergeometric}. 
These families have numerous important applications in many areas of Mathematics, Physics, and Engineering.
The name ``Hypergeometric'' comes from the fact that these polynomial families can be explicitly described in terms of some hypergeometric series  $_pF_q$ that uses the degree $n$ of the polynomial as a parameter.

There is an old idea to give some kind of classification of all possible hypergeometric families of orthogonal polynomials. 
The difficulty is that the condition of ``hypergeometricity'' mentioned above is difficult to formalize, so it is not clear what to classify.
 \medskip
 
In this paper we initiate a new approach to the classification. Namely, we do the following:
\begin{enumerate}[(i)]
\item We define a class of polynomial families that we call HG--families, thus suggesting a formalization for the intuitive term ``hypergeometric''. 
\item We classify all HG-families that are orthogonal with respect to some positive measure on the real line, and we show that each such family can be explicitly expressed in terms of some hypergeometric series $_pF_q$ in a precise way.
{\bf This is the main part of the paper.} 
\item In fact, we classify a wider class $Q$ of \emph{quasi-orthogonal} HG-families, by which we mean families that are orthogonal with respect to some non-degenerate complex-valued quadratic form.  

\end{enumerate} 

\DimaM{
We divide the class $Q$ into three components $W, E, F$, of dimensions 6, 3, 3, respectively. 
The component $W$ contains \emph{all}\footnote{See Remark \ref{rem:class} regarding the Hermite family.
} infinite polynomial families that appear in the Askey scheme of orthogonal polynomials (see for example \cite{KLS}). The  components $E$ and $F$ correspond to two completely new families of hypergeometric orthogonal polynomials, which are expressible in terms of a $_4F_1$ series. They turn out to be related to the Lommel polynomials that appear in the study of Bessel functions.
}
\begin{remark}
In fact in \cite{KLS}  they consider a more general Askey scheme that also deals with $q$-hypergeometric functions and 
 with ``discrete families''  of orthogonal polynomials. We are planning to \DimaN{extend} our results to these cases in the forthcoming paper.
\end{remark}

 A hypergeometric function $F(A, B; x)$ in fact depends on two families of parameters, that we call $A$- parameters and $B$- parameters (see e.g. \eqref{=pFq} below). In fact, it is more natural to consider 
 the variable $x$ as one of $A$- parameters, and write $F(A,B)$.
 The main hypergeometric condition for a family of    polynomials $P_n(z)$ is that there exists a hypergeometric function $F$ such that $P_n(z) = F(A(n,z), B(n,z))$,
 where $A$ and $B$ are some simple functions of $n$ and $z$.
 Our definition of $HG$-families implies that the parameters $B$ do not depend on $n$ and $z$.
 
We also  define more general class  of \tb{``rational HG-families"}, for which the $B$-parameters may also depend on $n$.     Using the same technique, we have made a first step towards the classification of such families- see Theorem \ref{thm:rat} below.
We also give some very interesting examples of rational HG-families of orthogonal polynomials 
  that are explicitly described in terms of hypergeometric functions, but are quite different from all classical families. 
  
 Our methods are purely algebraic.  
 In particular, we first describe $Q$, and this explicit description of $Q$ is valid over every base field of characteristic zero. 
The class of orthogonal families is a subset of $Q$, and its description is deduced using the Gauss-Favard theorem. 
 
 \begin{remark}
Our result is distinct from and more general than previous classifications, which require the polynomials to be eigenfunctions of a specific differential or difference operator. See section \ref{subsec:backgr}.
\end{remark}



Let us now define the class of HG-families and state our main results.

%
%
%


\subsection{HG families}

We consider sequences  
 $ \calP = (P_0, P_1, \ldots, P_n,\ldots )$, where each $P_n=P_n(z)$ is a monic polynomial of degree $n$ in $\C[z]$. We will call such a sequence a \tb{family of polynomials}.

Usually, hypergeometric series are described in terms of the monomial basis for the space of polynomials. We will work with more general bases, called \emph{Newtonian bases}.
A Newtonian basis is a basis defined by a recurrence relation given by a polynomial in the index.

\begin{defn}
Fix a polynomial $R\in \C[s]$. The \emph{Newtonian basis} given by $R$ is the family of polynomials $\Phi_R=\{\Phi_k\}$ defined by the recurrence relation
\begin{equation}\label{=Phi}
\Phi_{k+1}=z\Phi_k+R(k)\Phi_k \text{ for }k\geq 0, \text{ and } \Phi_{-1}=0,\,\,\Phi_{0}=1.
\end{equation}
\end{defn}
 
Once a Newtonian basis $\{\Phi_k\}$ is fixed, to give a family of polynomials is the same as to give a system of coefficients $c(n,k)$ such that $P_n = \sum_{k=0}^n c(n, k) \Phi_k$ for all $n\geq 0$.
\DimaM{Since our families are monic, we have $c(n,n)=1$. We also define
\begin{equation}\label{=cnk}
c(n,k)=0 \text{ unless }n\geq k\geq 0. 
\end{equation}

Roughly speaking, we will say that a family $\cP$ is a \emph{rational HG-family} if the ratio of subsequent coefficients $c(n,k)$ of each polynomial $P_n$ is a rational function in both indices. That is,
if there exists a rational function $f(u,s)$  such that 
 $c(n, k)$ satisfy
\begin{equation}\label{=cnkf}
c(n,k+1)=f(n,k)c(n,k). 
\end{equation}
Note that since $c(n,n)=1$, by this condition $f$ defines the coefficients $c(n,k)$ recursively, and thus the pair $(R,f)$ defines the family. 
We will say that a rational HG-family is an \emph{HG-family} if $f(u,s)$ depends polynomially on the first variable $u$.
Since $f$ can have poles we need to be more careful.

\begin{defn}\DimaK{Let $\calP=\{P_n\}$ be a family of  polynomials. We say that $\calP$ is a \tb{rational HG-family} if there exist relatively prime polynomials $N(u,s)$ and $D(u,s)$,  and a Newtonian basis $\Phi_R=\{\Phi_k\}$, such that for any $n\in \Z_{\geq 0}$ we have $P_n=\sum_{k=0}^n c(n,k)\Phi_k$ where $c(n,k)$ satisfy 
\begin{equation}\label{=ND}
D(n,k)c(n, k + 1)  =  N (n, k)c(n, k) \quad\text{ for all } n,k \in \Z.
\end{equation}
We say that $\calP$ is an \tb{HG-family} if it is a rational HG-family, and $D$ does not depend on $u$, {\it i.e.} $D\in \C[s]$.}

We will say that a rational HG-family $\cP$ is given by the pair $(R,f)$ where $R$ is the polynomial defining the Newtonian basis, and $f$ is the rational function
$f(u,s)=N(u,s)/D(u,s)$. 
\end{defn}
Note that the condition $\deg P_n=n$ and (\ref{=ND}) imply that $N(u,s)$ is divisible by $u-s$. For HG-families, the condition $c(n,-1)=0$ implies that  $D(u,s)$ is divisible by $s+1$.
Let us  now give several examples  of HG-families. 

\begin{example}
\begin{enumerate}[(i)]
\item If we take $R=0$, $N=(s-u)(u+s)$, and $D=s+1$ we obtain the Bessel family, given by the closed formula 
$$P_n=\frac{(-1)^n}{\binom{2n-1}{n}}\sum_{k=0}^n\binom{n+k-1}{k,k-1,n-k}z^k$$
\item If we take $R=0$, $N=u-s$, and $D=u-s-1$ we obtain the family $P_n=z^n$.
Indeed, in this case we have $\Phi_k=z^k$, and $c(n,k)=\delta_{nk}$ for $n\geq 0$. 

 Another way to obtain the same family is to take $R=1$, $N=u-s$, and $D=s+1$. Then we have $\Phi_k=(z-1)^k$, $c(n,k)=\binom{n}{k}$, and $P_n=\sum_{k=0}^n \binom{n}{k}(z-1)^k=z^n$.

\item If we take $R(s)=-s$, $N=s-u$, and $D=s+1$ we obtain the Charlier family, given by the closed formula 
$$P_n=\sum_{k=0}^n(-1)^{n-k}\binom{n}{k}z(z-1)\cdots(z-k+1)$$
\end{enumerate}
\end{example}
}

\subsection{Quasi-orthogonal and orthogonal families}

We will say that a family $\calP$ of polynomials is \tb{quasi-orthogonal}  if there exists a linear functional $M: \bC[z] \to \bC$ such that the matrix with entries $A_{ij} = M(P_i P_j)$ is an invertible diagonal matrix -- 
in other words, we have $A_{ij} = 0$ if $i \neq j$ and $A_{ij} \neq 0$ if $i = j$.

\sms

The functional $M$ is usually called a \emph{moment functional}. It is easily seen that $M$ completely determines $\calP$. \DimaD{Conversely, the functional $M$ is uniquely determined by the quasi-orthogonal family $\cP$ and the normalization $M(1)=1$.}
\sms

We say that a quasi-orthogonal family $\calP$ is \tb{orthogonal} if all the polynomials $P_i$ are defined over $\bR$ and there exists a moment functional $M$ such that all
the diagonal entries $A_{ii}$ are real and strictly positive.
It is known that in this case the moment functional $M$ is defined by integration with some positive measure $\mu$ on $\bR$ that has an  infinite support, and $\calP$ is, up to normalization, the standard
family of orthogonal polynomials defined by this measure (see {\it e.g.} \cite[Ch.II]{Chi} or  \cite[Ch.2]{Ism}).

\subsection{Main results}

\begin{introtheorem}[See \S \ref{sec:PfMain} below]\label{thm:main}
Suppose that a pair $(R,f)$ defines a quasi-orthogonal HG-family. Then 
the pair $(R,f)$ belongs to the following list
\begin{enumerate}[(a)]
\item \label{it:W} $R$ has degree $\le 2$, and there exist polynomials $q,v$ of degree  $\le 1$ such that 
\begin{equation*}
f(u,s)=\frac{(s-u)q(s+u)}{sq(s)\big[R(s+1)-R(0)\big] + (s+1)v(s)}
\end{equation*} 
\item \label{it:E} $R$ is a constant, and for some constants $c,\lam$ we have
\begin{equation*}
f(u,s)=\lam(s-u)(s-u+1-c)\frac{(s+u+1)(s+u+c)}{(s+1/2)(s+1)}.
\end{equation*}
\item \label{it:F} $R$ is a constant, and for some constants $c,\lam$ we have
\begin{equation*}
f(u,s)=\lam(s-u)(s-u+1-c)\frac{(s+u+2)(s+u+c+1)}{(s+1)(s+3/2)}.
\end{equation*}
\end{enumerate}
\end{introtheorem}

\DimaN{This theorem  implies in particular that $N(u,s)$ factorizes into linear factors over $\C$.}
Most choices of $(R,f)$ from the above list lead to a well-defined quasi-orthogonal family, but not every choice does. Also, different choices of $(R,f)$ can lead to the same family. \DimaN{Our next two results will clarify the situation. In order to formulate them, we first recall the classical result of Gauss and Favard.}


\begin{thm}[{see {\it e.g.}\cite[\S 4]{Chi}}]\label{thm:Fav}
Let $\{P_n\}_{n=0}^{\infty}$ be a family of  polynomials.
\begin{enumerate}[(i)]
\item The family $\{P_n\}$ is quasi-orthogonal if and only if
there exist (unique) sequences $\alp_n$ and $\beta_n$ of complex numbers such that
\begin{equation}\label{=3t}
zP_n=P_{n+1}+\alp_nP_n+\beta_{n-1}{P_{n-1}}  \,\,\forall  n\in \Z_{\ge 0}
\end{equation}
and $\beta_n\neq 0$ for all $n\in \Z_{\geq 0}$.
\item A quasi-orthogonal family $\{P_n\}$ is orthogonal if and only if   for all $n\in \Z_{\ge 0}$ we have $\alp_n,\beta_n\in \R$, and $\beta_n>0$.
\end{enumerate}
\end{thm}
Since all $P_n$ are monic, we have $P_0=1$. Also, 
we take $P_{-1}= 0$ by definition, and thus for $n=0$ \eqref{=3t} reads $z=P_{1}+\alp_0$.
In these terms we can now formulate our next result, but first we will need some notation. Given a rational function $f$ in two variables, 
define
\begin{equation}\label{=f12}
f_1(u):=f(u+1,u)^{-1} \text{ and } f_2(u):=f(u+2,u)^{-1}f_1(u+1).
\end{equation}

\begin{introprop}[Appendix \ref{subsec:Jacrat}]\label{th:JacAlp}
Let $\{P_n\}_{n=0}^{\infty}$ be an HG family given by a pair
 $(R,f)$. Let $f_1$ and $f_2$ be as in \eqref{=f12}, and define rational functions  in  $\alp$ and $\beta$ in one variable $u$ by
\begin{align*}
\alp(u)=f_1(u-1)-f_1(u)-R(u), \quad 
\beta(u)=f_2(u-1)-f_2(u)-f_1(u)\left(\alp(u+1)+R(u)\right)
\end{align*}
Suppose that $f_1$ and $f_2$ have no poles in non-negative integers, and that $(R,f)$ is as in one the cases (\ref{it:W}-\ref{it:F})
of  Theorem \ref{thm:main}.  Then the family $\{P_n\}_{n=0}^{\infty}$ satisfies the recursion \eqref{=3t} with sequences $\alp_n$ and $\beta_n$ given by  $\alpha_n=\alpha(n)$ and $\beta_n=\beta(n)$ for all $n>0$, and 
 $\alpha_0=\alp(0)-f_1(-1)$, $\beta_0=\beta(0)-f_2(-1)$.
\end{introprop}

The function $f_1$ may have a pole at $-1$. However, \DimaN{in this case $\alp$ also has a pole at $0$, and our expression for $\alp_0$ should be interpreted as $\alp_0=-f_1(0)-R(0)$, which is well-defined}. A similar remark is in place for $\beta_0$.

Theorem \ref{thm:Fav} and Proposition \ref{th:JacAlp} imply the following corollary.

\begin{introcor}\label{cor:ab} Let $\cP=\{P_n\}_{n=0}^{\infty}$ be an HG family given by a pair
 $(R,f)$. Suppose that $(R,f)$ is as in one of the cases (\ref{it:W}-\ref{it:F}) of  Theorem \ref{thm:main}.  Define sequences $\{\alp_n\},\{\beta_n\}$ as in Proposition \ref{th:JacAlp}. Then
\begin{enumerate}[(i)]

\item \label{it:abqort}  The family $\cP$ is well-defined and quasi-orthogonal if and only if for  all integer $n\geq 0$ we have that $f_1(n)$ and $f_2(n)$ are finite, and in addition
$\beta_n\neq 0$. 

\item The family $\cP$ is well-defined and orthogonal if and only if  for  all integer $n\geq 0$ we have that $f_1(n)$ and $f_2(n)$ are finite, and in addition $\alpha_n,\beta_n\in \R$ and $\beta_n>0$.

\item \label{it:abchoice} Two choices of $(R,f)$ in Theorem \ref{thm:main} lead to the same family $\calP$ if and only if they yield the same values for $\alpha_n,\beta_n$  for all $n\geq 0$.
\end{enumerate}
\end{introcor}

In Theorem \ref{thm:Classmain} below we define all quasi-orthogonal HG families explicitly in terms of hypergeometric functions, and also give an explicit description of all possible values of the parameters.

\subsubsection{Results  on rational HG families}

\begin{introtheorem}[See \S \ref{sec:PfMain} below]\label{thm:rat}
Let $\calP$  be a quasi-orthogonal rational HG family. Then
there exist   $g(u,s)\in \C(s)[u]$, $f\in \C(u,s)$, and $R\in \C[s]$ such that $\cP$ is given by  $(R(s),f(u,s)g(u,s+1)/g(u,s))$, 
and  one of the following holds.

\begin{enumerate}[(a)]
\item \label{it:rat:a} The pair $(R,f)$ satisfies the conditions of \eqref{it:W} in Theorem \ref{thm:main}.
\item \label{it:rat:b}$R$ is a constant, and 
$$f(u,s)=\frac{s-u}{s+1}\cdot\frac{(s-u-b)q(s+u)}{s+d},$$
for some  quadratic polynomial $q\in \C[t]$ and scalars  $b,d\in \C$.
\end{enumerate}
\end{introtheorem}

This theorem leaves two questions open. One is - what are all the possible constants $b,d$ and quadratic polynomials $q\in \C[t]$ in case \eqref{it:rat:b}?
In all the examples we know, $f$ satisfies the conditions of \eqref{it:E} or \eqref{it:F} of Theorem \ref{thm:main}. 

The second question is: given a pair $(R,f)$ as in \eqref{it:W}-\eqref{it:F} of Theorem \ref{thm:main}, for what $g\in \C(s)[u]$ the rational HG family defined by the pair $(R(s),f(u,s)g(u,s+1)/g(u,s))$ is also quasi-orthogonal?  We should remark that such examples are rare, and we describe some in Appendix \ref{app:gauge}. 

In particular, we construct a quasi-orthogonal two-parameter rational HG-family $\calP^{(c,\lam)}$   that interpolates between the families $E^{(c+1)}$ and $F^{(c)}$, and is not an HG family for $\lam\notin\{0,1\}$. We also construct a two-parameter quasi-orthogonal rational HG family  that interpolates between two families of Jacobi polynomials, and give a subset of parameters for which the family is orthogonal and is given by a hypergeometric function. 
Let us describe this one-parameter subfamily $R_n^{\lam}$. 
Let $Q_n$ be the Jacobi family with $a=1, b=3/2$, and  $g(u,s):=\lam (u+1/2) +s +1/2$, where $\lam\notin \{-1,1\}$. Then 
\begin{equation}
R^{\lam}_n:=g(n,z\partial_z)Q_n=c_n \,_3F_2(-n,n+1,\lam n + \frac{\lam +3}{2}; 3/2,\lam n + \frac{\lam +1}{2} ;z),
\end{equation}
where $c_n=(\lam+1)(n+1/2)$. See \eqref{=pFq} below for the notation $\,_3F_2$.
This is a rational HG-family, it is orthogonal, and satisfies a third order differential equation. 
We also construct another rational HG-family obtained that interpolates between two continuous Hahn families.
For further details on this family we refer the reader to \S \ref{subsec:diff}.

\subsection{Background and related results}\label{subsec:backgr}
We refer the reader to \cite{WW, Wat} for the classical theory \DimaN{of hypergeometric orthogonal polynomials} and to \cite{BP_Book, AAR, GR, Ism, KWKS, KLS} for some more recent developments.

The classification problem for such polynomials has a long history. It is typically framed through the study of polynomial eigenfunctions associated with a second-order differential or difference operator $L$. In this framework, their orthogonality arises naturally from the self-adjointness of $L$, consistent with classical Sturm-Liouville theory. The earliest work in this direction are due to Bochner \cite{Boch} and Hahn \cite{Hahn}, and these have been generalized considerably over the years by a number of authors, see for example \cite{Al, GH, VZ, VZ-N, V-S, Koo}.  

Some of the key results are summarized in \cite{Ism, KLS}. For example, the discussion in \cite{KLS} considers 10 distinct classes of operators. In each instance, once the general form of $L$ is fixed, the classification becomes a tame problem with finitely many parameters, and proceeding in this manner one eventually obtains a total of 44 families. The explicit form of the various $L$ implies that all 44 families are ``hypergeometric'' in nature, which means that they can be obtained by a suitable specialization of some hypergeometric series  $\,_pF_q$ or its $q$-analog $\,_r\phi_s$.  These families were organized by R.~Askey into a hypergeometric hierarchy, now known as the \emph{Askey scheme} \cite{KLS}.

This prompts the natural question of whether more intricate operators might yield additional hypergeometric families. More generally, one might seek to classify all hypergeometric and $q$-hypergeometric orthogonal families. In this paper, we classify a large class of orthogonal families, which {\it a priori} includes all families that might arise by a suitable specialization of hypergeometric series. Notably, we do not assume the former or even the existence— of a corresponding operator $L$, and in fact we obtain two new families, which are not in the Askey scheme. The $q$-hypergeometric setting will be addressed in a subsequent paper.

\subsection{Relation to classical families}\label{subsec:class}
Recall the definition of the hypergeometric series
\begin{equation} \label{=pFq}
_pF_q(\underline{a};\underline{b};z)
=\sum_{k\ge 0}
\frac{(a_1)_k \cdots (a_p)_k}{(b_1)_k \cdots (b_q)_k k!} z^k,
\quad (c)_k:=c(c+1)\cdots (c+k-1),
\end{equation}
and note that if $a_1=-n$ then the infinite series \eqref{=pFq} truncates to give a polynomial of degree $\le n$.
Theorem \ref{thm:main} implies that every quasi-orthogonal HG-family can be described using this type of special case of hypergeometric series. More precisely, we have the following theorem.

\begin{introtheorem}[See \S \ref{sec:PfClassmain} below]\label{thm:Classmain}
An HG family  $\{P_n\}_{n=0}^{\infty}$ is quasi-orthogonal if and only if it  arises by a \emph{rescaling} $P_n(z)\mapsto P_n(\lam z)$ and/or a \emph{renormalization} $P_n (z)\mapsto \lam_nP_n(z)$ and/or \emph{shift} $P_n(z)\mapsto P_n(z+\lam)$ from one of the hypergeometric series listed below. In these series, the parameters $a,b,c,d$ are complex numbers.
\begin{enumerate}[(a)]
\item \label{it:AJac} Jacobi: $_2F_1(-n,n+a;b;z)$ \quad with: $a,b-1,a-b \notin \Z_{< 0}$;\quad  
\item \label{it:Bes} Bessel: $_2F_0(-n,n+a;;z)$ \quad  with $a\notin \Z_{< 0}$;
\item \label{it:ALag} Laguerre: $_1F_1(-n;b;z)$ \quad  with $b\notin \Z_{\leq 0}$;

\item \label{it:AWil} Wilson: $P_n(z^2)=\,_4F_3(-n,n+a+b+c+d-1,z,2a-z; a+b, a+c,a+d;1)$\\ with $a+b,\,a+c,\,a+d,b+c,b+d,c+d,\,a+b+c+d\notin \Z_{\leq 0}$;
\item \label{it:AcdH} Continuous dual Hahn: $P_n(z^2)=\,_3F_2(-n,z,2a-z; a+b, a+c;1)$ with $a+b,a+c,b+c\notin\Z_{\leq 0}$;
\item \label{it:AcH} Continuous Hahn: $P_n(z)=\,_3F_2(-n,n+b+c+d-1, z; c, d;1)$\\
with $c,d,b+c,b+d,c+d,b+c+d\notin \Z_{\leq 0}$;
\item \label{it:AMei} Meixner:  $P_n(z)=\,_2F_1(-n, z; b;1-1/c)$ with $c\notin \{0,1\}, b\notin \Z_{\leq 0}$;
\item \label{it:ACha} Charlier: $P_n(z)=\,_2F_0(-n, z; -;-1/a)$ with $a\neq 0$;

\item \label{it:AE} $E_n^{(c)}$:  $_4F_1(-n,-n-c+1,n+c,n+1; 1/2;z)$\quad  with  $c\notin \Z_{\le 0}$;
\item \label{it:AF} $F_n^{(c)}$: $_4F_1(-n,-n-c+1,n+c+1,n+2; 3/2;z)$\quad  with $c\notin \Z_{\leq0}$;
\end{enumerate}
\end{introtheorem}

Here, cases (\ref{it:AJac}-\ref{it:ACha}) correspond to case \eqref{it:W} of Theorem \ref{thm:main}, and cases (\ref{it:AE},\ref{it:AF}) correspond to cases (\ref{it:E},\ref{it:F}) of Theorem \ref{thm:main}. 
The following table describes how case \eqref{it:W} of Theorem \ref{thm:main} decomposes to cases (\ref{it:AJac}-\ref{it:ACha})  of Theorem \ref{thm:Classmain}, depending on the degrees of $R,q$, and $v$, where we take $\deg(0)=-\infty$:

\begin{equation}\label{tab:decomp}
\begin{tabular}{c c c c} 
\hline 
{\normalsize family}  &  $\deg q$ &$\deg v$ & $\deg R$ \\  
\hline 

{\normalsize Jacobi} &  1 &1 & $\leq 0$   \\   \hline

{\normalsize Bessel} &  1 &0 &  $\leq 0$   \\   \hline
{\normalsize Laguerre}  &  0 & 1 & $\leq 0$\\ \hline
{\normalsize Wilson} &  1 &$\leq 1$ & 2   \\   \hline
{\normalsize Continuous dual Hahn} &  0 &$\leq 1$ & 2   \\   \hline

{\normalsize Continuous Hahn} &  1 &$\leq 1$ & 1   \\   \hline

{\normalsize Meixner} &  0 &1 & 1   \\   \hline
{\normalsize Charlier} &  0 &0 & 1   \\   \hline

\end{tabular}
\end{equation}


\begin{remark}\label{rem:class}
Let us explain how this table relates to the part of the Askey scheme given by the hypergeometric series (rather than basic hypergeometric series). This scheme includes the families in the table, and in addition the Hermite family and finite families. 
All the finite families are obtained by substituting in one of our families special values of parameters that are not allowed in the theorem. 

Let us remark that classically the Bessel family is viewed as a finite family, since it is not orthogonal  for any parameter $a$. However, it is quasi-orthogonal for all $a\notin\Z_{< 0}$. 

The Hermite polynomials are alternatively even and odd. Thus one can view them as a combination of two families, each of which happens to be a Laguerre family. In this sense this family is covered, though the full Hermite family does not satisfy our HG condition.

\end{remark}


\subsection{Extension to other base fields}

Our definitions of HG families and quasi-orthogonal families make sense over every field $F$.
If $F$ has characteristic 0 then our results hold for it, as the following theorem states.
\begin{introthm}[\S \ref{sec:allFields}]\label{thm:allFields}
Theorems \ref{thm:main}, \ref{thm:rat}, and \ref{thm:Classmain}, Proposition \ref{th:JacAlp}, and parts \eqref{it:abqort} and \eqref{it:abchoice} of Corollary \ref{cor:ab} hold over every field $F$ of characteristic 0.
\end{introthm}

\subsection{Main ideas of our proofs}

The technical heart of our paper consists of obtaining \emph{upper bounds} for the degrees of $R(s)$ and $f(u, s)$ for every family of HG type, which we achieve in two steps. First, to each family we associate a module for a certain algebra $A$, which is an extension of the field $K=\C(u,s)$ by the shift operators $u \mapsto u\pm1$ and $s \mapsto s \pm 1$. By construction this module has $K$-dimension at most two, but we show that it is, in fact, one-dimensional.

One-dimensional $A$-modules have been classified in \cite{Ore}. 
He provided a very explicit model for every such module - each is isomorphic to the module generated by a meromorphic function of a very special form.
Applying this result to our setting we deduce that $f(u,s)$ must factor completely into linear factors of a very special kind. 
Then we use the three-term recursion to bound the number of factors, and thus obtain the desired degree bounds. 

The bounds imply that the families of HG type depend on finitely many parameters, namely, the (finitely many) coefficients of $R(s)$ and  $f(u,s)$. In other words, the classification problem is \emph{tame}, and we proceed to carry it out explicitly.

\subsection{Structure of the paper and the scheme of the proofs}

In \S \ref{sec:prel} we describe the classification of one-dimensional $A$-modules from \cite{Ore} (see Theorem \ref{thm:mod} below).

In \S \ref{sec:PolMod} we construct an  $A$-module from the coefficient function $c(n,k)$ and prove that it is one-dimensional. \DimaK{We first state that $f$ determines the family $P_n$, and that  the sequences $\alp$ and $\beta$ coincide with rational functions, except possibly for finitely many $n$. 
Then we consider} the space of all $\C$-valued functions on $\Z^2$ as a module over $\C[u,s]$, and extend scalars to obtain a vector space over the field of rational functions $K=\C(u,s)$.
This vector space has a natural structure of an  $A$-module.
Then we consider the submodule $M$ generated by \DimaG{the image of $c(n,k)$, that we denote by $w$}. By construction we have $S^{-1}w=(S^{-1}f)w$. By the Gauss-Favard theorem (Theorem \ref{thm:Fav}) we have the three-term recursion
\begin{equation}\label{=3tI}
zP_n=P_{n+1}+\alp_nP_n+\beta_{n-1}{P_{n-1}}  \,\,\forall  n\in \Z_{\ge 0}
\end{equation}
It gives a linear relation on $S^{-1}w, \, Uw, \, w,$ and $U^{-1}w$ over $K$. Using this relation and the commutation $US=SU$ we prove in Theorem \ref{thm:1dim} below that $\dim_K M=1$.

In \S \ref{sec:3t} we use the explicit realization from Theorem \ref{thm:mod} to almost classify all pairs $M,v$, where $M$ is a one-dimensional $A$-module, and $v$ is a vector satisfying a relation of the form \eqref{=3tI}. Namely, we realize $v$ as a function $$\Phi=\exp(cu+ds)g(u,s)\prod \Gamma(k_iu+l_is+c_i),$$ where $g\in K=\C(u,s)$, satisfying
\begin{equation}\label{=3PhiIntro}
S^{-1}\Phi = U\Phi +\alp \Phi + \beta U^{-1}\Phi.
\end{equation}
For every $i$, $\Phi$ has an infinite family of poles in the lines $k_iu+l_is+c_i = n$ for $n\in \Z_{\geq 0}$. Any of the poles of one of the terms in \eqref{=3PhiIntro} has to be a pole of at least one other term. From this we deduce that $k_i,l_i\in \{-1,0,1\}$. Further analyzing the poles we obtain that there exist polynomials $p,q,w\in \C[t]$ and $g\in \C[u,s]$ such that
\begin{equation}\label{=fIntro}
f=\frac{S\Phi}{\Phi}=\frac{q(u+s)p(s-u)}{w(s)}\cdot \frac{Sg}{g}
\end{equation}

In \S \ref{sec:cases} we use elementary algebraic considerations to deduce from \eqref{=3PhiIntro} strong restrictions on $p,q$, and $w$ (see Propositions \ref{prop:Jac} and \ref{prop:cases}, and Corollary \ref{cor:pqDet}). In particular, we show that $\dim p, \dim q,\dim w\leq 2$.

\DimaK{In \S \ref{sec:PfRat} we deduce Theorem \ref{thm:rat} from \eqref{=fIntro} and Corollary \ref{cor:pqDet}. 

In \S \ref{sec:PfMain} we prove Theorem \ref{thm:main}, Proposition \ref{th:JacAlp} and Corollary \ref{cor:ab}.
We deduce Theorem \ref{thm:main} from  Theorem \ref{thm:rat} and Corollary \ref{cor:pqDet}. To do that we
 show that the condition that in HG-families $f$ depends polynomially on $u$ implies  that in \eqref{=fIntro} we can assume that $g=1$.

In \S \ref{sec:PfClassmain} we deduce Theorem \ref{thm:Classmain} from Theorem \ref{thm:main}. Most of the section is devoted to the computation of the exact areas of allowed parameters, namely those satisfying the conditions $N(n+2,n)\neq 0$, $N(n+1,n)\neq 0$, $D(k)\neq 0$ and $\beta_n\neq 0$.

In \S \ref{sec:new} we investigate the two new families $E_n^{(c)}$ and $F_n^{(c)}$,
\DimaC{express them through Lommel polynomials, and find the discrete measures on the real line with respect to which they are orthogonal. We also give the 4th order differential equations that they satisfy.}

In Appendix \ref{sec:allFields} we extend our results to all base fields of characteristic zero, in the following way. We first state that the results of \S \ref{sec:PolMod} hold over every field. These results reduce the problem to a statement on 1-dimensional modules with special elements. Every such pair of a module and an element involves only finitely many elements of the field, and thus is defined over a subfield of finite transcendence degree. This subfield then can be embedded into the field of complex numbers, and therefore the analysis of \S\S \ref{sec:3t}-\ref{sec:PfClassmain} holds for it.

\DimaD{ In Appendix \ref{sec:rat} we prove the technical lemmas from \S \ref{sec:PolMod}.} In particular, we prove that $f$ determines the family $P_n$, and that the sequences $\alp_n$ and $\beta_n$ coincide with rational functions, except possibly for finitely many $n$. In \S \ref{subsec:Jacrat} we prove Lemma \ref{lem:fxy} on HG-families.  

\DimaI{In Appendix \ref{app:gauge} we construct a two-parameter rational HG-family $\calP^{(c,\lam)}$ that interpolates between the families $E^{(c+1)}$ and $F^{(c)}$, and is not of an HG-family for $\lam\notin\{0,1\}$. We also construct a two-parameter family of rational type that interpolates between two families of Jacobi polynomials.

}
}

\section{Preliminaries}\label{sec:prel}
Recall the notation $K=\C(u,s)$ and
$A:=$ the subalgebra of $\End_\C(K)$ generated by $K$ and the shift operators  $U^{\pm 1}:u\mapsto u\pm1;\, S^{\pm1}:s\mapsto s\pm1$. The $S^{\pm 1},U^{\pm 1}$ act on $K$ by $(U^{\pm1}g)(u,s):=g(u\pm 1,s)$ and $(S^{\pm1}g)(u,s):=g(u,s\pm 1)$. We say that an $A$-module is \emph{one-dimensional} if it is one-dimensional as a vector space over $K$.

For every one-dimensional module $M$ and every non-zero vector $v\in M$ there exist unique $f,h\in K$ such that $Sv=fv$ and $Uv=hv$. It is easy to see that
\begin{equation}\label{=main}
Uf/f=Sh/h
\end{equation}
and that if we replace $v$ by $gv$ for some non-zero $g\in K$, the pair $(f,h)$ will change to $(fSg/g,hUg/g)$. We will call such pairs equivalent. Conversely, it is easy to see that every pair $(f,h)\in K^{\times}\times K^{\times}$ satisfying \eqref{=main} defines a one-dimensional $A$-module.
The set of isomorphism classes of one-dimensional $A$-modules forms a group under tensor product over $K$, where the $U$ and the $S$ act on a product by acting on both factors. This corresponds to elementwise products of pairs $(f,h)$. This group was computed in \cite{Ore}. To formulate this result, note that the field of meromorphic functions in $u,s$ has a natural $A$-module structure.
\begin{defn}
A meromorphic function $\eta$ in two variables $u,s$ is of $\gamma$ type if it is a product of the form  \begin{equation}\label{=eta}
\exp(au+bs)\prod_{i=1}^n \Gamma(k_i u+l_i s+c_i),\quad a,b,c_i \in \C, \; (k_i,l_i) \text{ coprime in } \Z^2,
\end{equation}
where $\Re c_i\in [0,1)\, \forall i$,  and if $(k_i,l_i)=(-k_j,-l_j)$ then $c_i\neq 1-c_j$.
\end{defn}

It is easy to see that for every $\eta$ of $\gamma$ type, the space $K\eta$ spanned by it is invariant under $U^{\pm 1}$ and $S^{\pm 1}$, and thus is an $A$-module.

\begin{theorem}[\DimaB{\cite{Ore}, see also \cite[Proposition 1.7]{Sab}}]\label{thm:mod} The correspondence $\eta\mapsto K\eta$ is a bijection between the set of functions of $\gamma$ type and the set of isomorphism classes of 1-dimensional $A$-modules.
\end{theorem}

\begin{rem}
\begin{enumerate}[(i)]
\item \DimaF{The formulations of the theorem in \cite{Ore} and in \cite[Proposition 1.7]{Sab} are slightly different from ours. For example, in \cite[Proposition 1.7]{Sab} negative powers of $\Gamma(k_i u+l_i s+c_i)$ are allowed. On the other hand, \cite[Proposition 1.7]{Sab} reduces the set of allowed pairs $(k_i,l_i)$ to include exactly one of $\{(k,l),(-k,-l)\}$ for any coprime $(k,l)\in \Z^2$. The equivalence of the two formulations follows from the Euler's reflection formula:
\begin{equation}\label{=Eu}
\Gamma(x)\Gamma(1-x)=\frac{\pi}{\sin(\pi x)}=\frac{2\pi i}{1-\exp(-2\pi i x)}\exp(-i\pi x)=\phi(x) \exp(-i\pi x),
\end{equation}
where $\phi$ is a periodic function with period 1. Thus multiplication by $\phi(ku+ls+c)\exp(-i \pi c)$ defines an isomorphism between $\langle \Gamma(ku+ls+c)^{-1}\rangle$ and $\langle \Gamma(-ku-ls+1-c)\exp(i\pi ku+i\pi ks) \rangle$.}

\item The Euler's reflection formula \eqref{=Eu} allows to define a group structure on the set of functions of $\gamma$ type by defining the product of two such functions $\eta$ and $\xi$ to be the only function $\theta$ of $\gamma$ type such that the function $\phi(u,s):=\eta(u,s)\xi(u,s)/\theta(u,s)$ satisfies $S\phi=U\phi=1$. Under this group structure, the bijection $\eta\mapsto K\eta$ becomes a group isomorphism.
\end{enumerate}
\end{rem}

\section{From quasi-orthogonal rational HG families  to 1-dimensional $A$-modules }\label{sec:PolMod}

Let $\{P_n\}$ be a rational HG family defined by a pair $(R,f)$ where $R\in \C[s]$ is a polynomial, and 
\begin{equation}
f(u,s):=\frac{N(u,s)}{D(u,s)}\in \C(u,s)
\end{equation}
is a rational function. Assume that the family  $\{P_n\}$  is quasi-orthogonal.

Recall that $R$ defines a Newtonian basis  $\{\Phi_k\}$ by the recurrence relation
\begin{equation}\label{=Phi2}
\Phi_{k+1}=z\Phi_k+R(k)\Phi_k, \quad \Phi_{-1}=0,\,\,\Phi_{0}=1.
\end{equation}
Then $\Phi_k$ is a monic polynomial of degree $k$, and therefore $\{\Phi_k\}$ is a basis to the space of polynomials.
Denote by $c(n,k)$ the coefficients of $P_n(z)$ in the basis $\Phi_k$. 
Note that $c(n,k)$ is a two-parameter sequence satisfying $c(n,k)=0$\ unless $n\geq k\geq 0$. By definition of a rational HG -family, we have 
\begin{equation}\label{=hyp0}
c(n, k + 1) D(n,k) = c(n, k) N (n,k)
\end{equation}
A priori, it is possible that for some $n\geq k\geq 0$ we have $N(n,k)=D(n,k)=0$.
Since $D$ and $N$ are co-prime, the set $X=\{(u,s)\in \C^2\, \vert \, N(u,s)=D(u,s)\}$ is finite.  
 For $(n,k)\in \Z^2\cap X$, \eqref{=hyp0} puts no restriction on $c(n,k)$.
 Other vanishings of $N(n,k)$ also limit the information \eqref{=hyp0} gives. However, the following lemma says that \eqref{=hyp0}, together with the condition that the family is quasi-orthogonal, uniquely defines the family. 




\begin{lem}\label{lem:RDNdef}
The polynomials $R(s), D(u,s),N(u,s)$ define the monic quasi-orthogonal family uniquely.
\end{lem}

We postpone the proof of this lemma, as most lemmas in this section, to Appendix \ref{sec:rat}.


\begin{lem}[Appendix \ref{sec:rat}]\label{lem:fxyRat}
There exist polynomials $x,y\in \C[u,s]$  such that 
\begin{equation}
N(u,s)=(s-u)x(u,s),  \quad D(u,s)=(s+1)y(u,s).
\end{equation}
\end{lem}

By the Gauss-Favard theorem (Theorem \ref{thm:Fav} above), there exist (unique) sequences $\{\alp_n\}_{n=1}^{\infty}$ and $\{\beta_n\}_{n=1}^{\infty}$ of complex numbers such that $\beta_n\neq 0$ for all $n\in \Z_{\geq 0}$ and 
\begin{equation}\label{=3tA}
zP_n=P_{n+1}+\alp_nP_n+\beta_{n-1}{P_{n-1}}  \quad\forall  n\in \Z_{\ge 0}
\end{equation}
For $n=0$, $\beta_{-1}$ is not defined, but $P_{-1}=0$, so the equation reads $z=P_1+\alp_0$. 
By the definition of $\Phi_k$ (see \eqref{=Phi2}) we have $z\Phi_k(z)=\Phi_{k+1}-R(k)\Phi_k$.
Thus 
\begin{multline}
zP_n(z)=\sum_{k=0}^n c(n,k)z\Phi_k(z)=\sum_{k=0}^n c(n,k)\Phi_{k+1}(z)-\sum_{k=0}^n c(n,k)R(k)\Phi_{k}(z)=\\=\sum_{k=0}^n (c(n,k-1)-R(k)c(n,k))\Phi_{k}(z)
\end{multline}

Thus considering the coefficient of $\Phi_k$ in \eqref{=3tA} we obtain
\begin{equation}\label{=3tc}
c(n,k-1)=c(n+1,k)+(\alp_n+R(k)) c(n,k)+\beta_{n-1} c(n-1,k)
\end{equation}

\begin{lem}[Appendix \ref{sec:rat}]\label{lem:rat}
There exist rational  functions $\alp$ and $\beta$ such that $\alp(n)=\alp_n$ and $\beta(n)=\beta_n$, except possibly for finitely many $n$. 
\end{lem}

Recall the notation $K=\C(u,s)$ and
$A:=$ the subalgebra of $\End_\C(K)$ generated by $K$ and the shift operators  $U^{\pm 1}:u\mapsto u\pm1;\, S^{\pm1}:s\mapsto s\pm1$. 
 
\begin{prop}\label{prop:1dim}
Let $M$ be an $A$-module generated by a non-zero vector $v\in M$. Suppose that there exist $\gamma,\delta,\lam\in K$ such that $S\gamma=\gamma, S\delta\neq \delta, Sv=\delta v,$ and 
\begin{equation}\label{=U20}
U^2v = \lambda Uv-\gamma v
\end{equation}
Then $\dim_KM=1$.
\end{prop}
\begin{proof}
By \eqref{=U20}, $v$ and $Uv$ span $V$, and thus $\dim_K V\leq 2$. Suppose by way of contradiction that
 $\dim_KV=2$. 
Then $v$ and $Uv$  are linearly independent.
Applying $S$ to \eqref{=U20} and using $US=SU$, $Sv=\delta v$, and $S\gamma=\gamma$ we have
\begin{equation}\label{=SU2}
SU^2v=S(\lambda U v)-S(\gamma v)=S(\lambda)U(\delta v)-\gamma\delta v=S(\lambda)U(\delta)U(v)-\gamma\delta v
\end{equation}
We also have
\begin{equation}\label{=SU22}
SU^2v=U^2(Sv)=U^2(\delta v) = U^2(\delta) U^2v =U^2(\delta)\lambda Uv-U^2(\delta)\gamma v
\end{equation}

Thus
\begin{equation}\label{=vUv}
S(\lambda)U(\delta)U(v)-\gamma\delta v=U^2(\delta)\lambda Uv-U^2(\delta)\gamma v
\end{equation}
Since $v,Uv$ are linearly independent, comparing their coefficients we have
\begin{equation}\label{=coeff}
-\gamma\delta =-U^2(\delta)\gamma \text{ and }
S(\lambda)U(\delta)=U^2(\delta)\lambda
\end{equation}
The first equation implies
$
U^2(\delta)=\delta
$, and thus $U\delta=\delta$, contradicting the assumption.
\end{proof}

Let $L$ be the $\C[u,s]$-module of all $\C$-valued functions on $\Z^2$, with the action given by restriction of any polynomial in $\C[u,s]$ to $\Z^2$, by substitution of $n$ for $u$ and $k$ for $s$. Let $N$ be the extension of scalars $N:=L\otimes_{\C[u,s]}\C(u,s)$.
The $S^{\pm 1},U^{\pm 1}$ act on $K$ and on $L$ by $(U^{\pm1}g)(u,s):=g(u\pm 1,s)$ and $(S^{\pm1}g)(u,s):=g(u,s\pm 1)$.
Let $S^{\pm 1},U^{\pm 1}$ act on $N$ by $S^{\pm1}(l\otimes g):=S^{\pm1}l\otimes S^{\pm1}g$ and $U^{\pm1}(l\otimes g):=U^{\pm1}l\otimes U^{\pm1}g$. This defines a structure of an $A$-module on $N$.
Let $w$ denote the image of the function $c(n,k)$ in $N$ and let $M$ be the $A$-submodule of $N$ generated by $w$.
\begin{lem}[Appendix \ref{sec:rat}]\label{lem:non0Sw}
We have $w\neq 0$ and $Sw=fw$.
\end{lem}

\begin{thm}\label{thm:1dim}
$\dim_KM=1$.
\end{thm}

\begin{proof}
 
We have 
\begin{equation}
zP_n=\sum c(n,k)z\Phi_k=\sum c(n,k)(\Phi_{k+1}-R(k)\Phi_k)=\sum c(n,k)\Phi_{k+1}-\sum c(n,k)R(k)\Phi_k
\end{equation}

Thus we can rewrite \eqref{=3tA} as 
\begin{equation}
\sum c(n,k)\Phi_{k+1}-\sum c(n,k)R(k)\Phi_k=\sum c(n+1,k)\Phi_{k}+\alp_n\sum c(n,k)\Phi_{k}+\beta_{n-1}\sum c(n-1,k)\Phi_{k}
\end{equation}

 Looking at the coefficient of $\Phi_k$ in \eqref{=3tA} we have
\begin{equation}
(S^{-1}f)^{-1}c(n,k)=(S^{-1}c)(n,k)=Uc(n,k)+(\alp(n)+R(k))c(n,k)+\beta(n-1)U^{-1}c(n,k) \end{equation}
for almost every $n\geq 1$. This also trivially holds for every $n\leq -2$. Thus, by multiplying both sides of the equality by a suitable polynomial in $n$, we can arrange for it to  hold for all $n$. Therefore, in the $A$-module $M\sub N=L\otimes_{\mathbb{C}[u,s]}K$ we have
\begin{equation}\label{=3tM}
(S^{-1}f)^{-1}w=S^{-1}w=Uw+(\alp+R) w+ U^{-1}\beta U^{-1}w
\end{equation}

Define $v:=U^{-1}w$ and $\gamma:=U^{-1}\beta$. From \eqref{=3tM} we have
\begin{equation}
U^2v = \lambda Uv-\gamma v,
\end{equation}
where $\lambda=S^{-1}f^{-1}-\alp-R\in K$.
Since $U$ and $S$ commute,  
$$Sv=SU^{-1}w=U^{-1}Sw=U^{-1}(f w)=(U^{-1}f)U^{-1}w= (U^{-1}f)v.$$ In other words $Sv=\delta v$, where $\delta=U^{-1}f\in K$.
Note that $\gamma$ does not depend on $s$, and thus $S\gamma=\gamma$.
By Lemma \ref{lem:fxyRat}, $f$ has the term $s-u$. Thus $Uf\neq f$ and $U\delta\neq \delta$. The theorem follows now from Proposition \ref{prop:1dim}.
\end{proof}

\begin{cor}\label{cor:ch}
There exists a (unique) rational function $h\in \C(u,s)$  such that 
outside of the zero set of a polynomial $r(n,k)$ we have 
\begin{equation}\label{=hc}
c(n+1,k)=h(n,k)c(n,k).
\end{equation}

\end{cor}
\begin{proof}
Since $\dim_K M=1$, there exists a (unique) rational function $h\in \C(u,s)$ such that $Sw=hw.$ Then we have
$$fShw=SUw=USw=hUfw,$$
thus $fSh=hUf$. 
By construction of $M$ this implies that \eqref{=hc} holds outside of the zero set of a polynomial $r(n,k)$. 
\end{proof}

\section{Classification of pointed 1-dimensional $A$-modules satisfying the 3-term relation}\label{sec:3t}

Let $\psi(u,s)$ be a meromorphic function of the form $g\phi$, where $\phi$ is of $\gamma$ type and $g\in \C(u,s)$. Then $S\psi=f\psi$ for some $f\in \C(u,s)$. Suppose also that $\psi$ satisfies the 3-term recurrence relation \eqref{=3tM}:
\begin{equation}\label{=3Phi}
S^{-1}\psi =  U\psi +(\alp+R) \psi +  U^{-1}(\beta \psi), \text{ where }\alp,\beta\in \C(u) \text{ with }\beta\not \equiv 0, \text{ and }R\in \C[s].
\end{equation}
By Proposition \ref{prop:1dim},
 $\psi$ generates a one-dimensional $A$-module. Thus $U\psi=h\psi$ for some $h\in \C(u,s)$. This $h$ necessarily satisfies
\begin{equation}\label{=main2}
Uf/f=Sh/h
\end{equation}

Since $S\psi=f\psi$, we have $S^{-1}\psi=S^{-1}(f)^{-1}\psi$, and \eqref{=3Phi} is equivalent to
\begin{equation}\label{=3tfh}
S^{-1}(f)^{-1} = h +\alp+R  + U^{-1}(\beta/h),
\end{equation}
We are looking for the space of joint solutions of \eqref{=main2} and \eqref{=3tfh}.
\begin{thm}\label{thm:3termForm}
Suppose that $U\psi=h\psi, \,\, S\psi=f\psi$, and $\psi\neq 0$ satisfies \eqref{=3Phi}. Then there exist polynomials $g\in \C[u,s]$ and $p,q,w \in \C[t]$, and a rational function $\sigma \in \C(u)$ such that
 \begin{equation}\label{=f}
f=\frac{q(s+u)p(s-u)Sg}{w(s)g}, \quad h= \frac{\sigma(u)q(s+u)Ug}{p(s-u-1)g}.
\end{equation}
Moreover, we can choose $p,q,w,$ and $g$ such that $p,q$ are monic,
$g$ is coprime to each of the polynomials $w(s\DimaI{-1}),q(u+s), p(s-u)$ and has no factors independent of $s$. The tuple $(p,q,w,\sigma,g)$ satisfying these conditions is unique, except that $g$ can be multiplied by a constant.
\end{thm}
To prove the theorem we will use the following proposition.
\begin{prop}\label{prop:inv}
Let $a$ be an irreducible monic non-constant polynomial in $s$ with coefficients in $\C(u)$.
If $S^kU^la=a$ and $(k,l)\neq (0,0)$ then $l\neq 0$ and
$a=s-(k/l)u+c$ with $c\in \C$.
\end{prop}
\begin{proof}
First of all $l\neq 0$ since otherwise $S^ka=a$, thus $a\in \C(u)$, and since $a$ is monic we have $a=1$, contradicting the assumption that $a$ is not constant. 
Let $d:=gcd(k,l), k':=k/d, l':=l/d.$ Let $m,n$ be such that $mk'+nl'=1$.
Now, perform a linear change of variables: $x:=l's-k'u$, $y:=ms+nu$.
Under this change, the shift operator in $y$ becomes $S^{k'}U^{l'}$.

The condition $S^kU^la=a$ implies that $a$ is invariant to shifts $y\mapsto y+d$. Thus $a$ depends only on $x$. Thus $a=pol(ls-ku),$ where $pol$ is a polynomial in one variable with complex coefficients. Since $a$ is irreducible, $pol$ is also irreducible and thus it must be linear. Since $a$ is monic in $s$,  we have $a=s-(k/l)u + c$.
\end{proof}
\begin{proof}[Proof of Theorem \ref{thm:3termForm}]
Since $\psi=g\phi$, where $\phi$ is of $\gamma$ type and $g\in \C(u,s)$, $\psi$ has  the form
$$\psi=\exp(cu+ds)g(u,s)\prod \Gamma(k_iu+l_is+c_i),$$
with $k_i,l_i\in \Z$ and $(k_i,l_i)$ coprime for every $i$, $\Re c_i\in [0,1)$ and \DimaF{if $(k_i,l_i)=(-k_j,-l_j)$ then $c_i\neq 1-c_j$. }
By removing the condition $\Re c_i\in [0,1)$ and using the identity $x\Gamma(x)=\Gamma(x+1)$ we assume without loss of generality that for every $i$, $k_iu+l_is+c_i$ is not a factor in the numerator of $g$, and $k_iu+l_is+c_i-1$ is not a factor in the denominator of $g$. 

We claim that for every $i$ there are only five possibilities for the pair $(k_i,l_i)$. Namely, $(k_i,l_i)\in \{(\pm 1,  1), (0, -1),(\pm1, 0)\}$ for every $i$. Indeed, fix $i$. If $l_i=0$ then $k_i\in \{\pm1\}$ since $(k_i,l_i)$ are coprime.  Thus we assume $l_i\neq 0$. Assume also, without loss of generality, that for every $j$ with $(k_j,l_j)=(k_i,l_i)$, we have $\Re c_j \leq \Re c_i$. Consider all $n\in \Z$ such that $\psi$ has a pole along the line  $k_iu +l_i s+c_i=n$. The maximal $n$ will be $0$. For $S^{-1}\psi$, $U\psi$, and $U^{-1}\psi$ such maximal $n$-s will be $l_i, -k_i$, and $k_i$ respectively.
By our assumptions on $\alp$ and $\beta$, they do not have  poles of this type.
Every pole of a term in \eqref{=3Phi} should occur in at least two terms. 
The terms $S^{-1}\psi,U\psi$, and $U^{-1}\psi$ enter with non-zero coefficients that cannot cancel the pole, and one of the numbers $k_i,-k_i$ is non-negative. Thus the two largest of the numbers 
$(0,l_i,k_i,-k_i)$ are equal. Thus either $k_i=0$ and $l_i<0$, or  $l_i=|k_i|$. Since $k_i$ and $l_i$ are coprime, this implies that $(k_i,l_i)\in \{(0,-1),(\pm 1,1)\}$.

The pairs $(\pm 1,1)$ give the factors $s\pm u + c_i$ in $f=S\psi/\psi$, the pairs $(\pm 1,0)$ have no effect on $f$, and the pair $(0,-1)$ produces the factor $(s-1+ c_i)^{-1}$ in $f$. The term $\exp(cu+ds)$ produces a multiplicative constant $\exp(d)$. 
The assumption that $k_iu+l_is+c_i-1$ is not a term in the numerator of $g$ for any $i$ implies that the numerator of $g$ has no common factors with $w(s\DimaI{-1}),q(u+s), p(s-u)$. 
Thus $f$ has the form as in \eqref{=f}, and by \eqref{=main2} so does $h$, except that $g$ is allowed to be a rational function.

It is left to prove that $g$ is a polynomial.
We can write 
\begin{equation}
\sigma=\delta_1/\gam, \quad \alp=\delta_0/\gamma, \quad U^{-1}(\beta/\sigma)=\delta_{-1}/\gamma,
\end{equation}
where $\gamma,\delta_1,\delta_{-1}\in \C[u]$ and $\delta_0\in \C[u,s]$ are polynomials that have no overall common factor.
By \eqref{=f}, equation \eqref{=3tfh} becomes
\begin{multline}\label{=3pq}
\gamma(u)w(s-1)S^{-1}(g)=\delta_1(u)q(s+u)q(s+u-1)Ug \,+\\ +(\delta_0(u)+R(s)\gamma(u))q(s+u-1)p(s-u-1)g+\delta_{-1}(u)p(s-u)p(s-u-1)U^{-1}g.
\end{multline}

Suppose, by way of contradiction, that $g$ is not a polynomial, 
and let $a$ be an irreducible factor of the denominator of (the reduced expression for) $g$. 
By our assumption on $g$, $a$ depends on $s$. 
Then $S^{-1}a$ is an irreducible factor in the denominator of $S^{-1}g$.
\DimaI{The assumption that 
$k_iu+l_is+c_i-1$ is not a factor in the denominator of $g$ for every $i$ implies that the denominator of $g$ is coprime to $w(s)$. Thus} $w(s-1)$ cannot cancel the factor $S^{-1}a$, and thus $S^{-1}a$ appears in the denominator of the reduced expression for the LHS of \eqref{=3pq}. Thus $S^{-1}a$ must appear in the denominator of the reduced expression for one of the terms of the RHS as well.
 Therefore $S^{-1}a=U^ib$, where $i\in \{-1,0,1\}$ and $b$ is another factor in the denominator of $g$. Then $S^{-1}b$ is also a factor in the denominator of $S^{-1}g$. Since the number of factors is finite, the factors obtained in this way must repeat. It follows that $S^{-k}U^la=a$, where $k,l\in \Z$, $k\geq |l|$, and $k>0$. 

Thus, by Proposition \ref{prop:inv}, $a=d(u+(l/k)s+c)$ for some $c,d\in \C$. Without loss of generality we assume $d=1$ and that among all the factors of the form $u+(l/k)s+e$ with the same ratio $l/k$, $a$ has minimal $\Re e$. The third term of the RHS of \eqref{=3pq} has $U^{-1}a = a-1$ in the denominator. Thus $a-1 \in \{S^{-1}b', b', Ub'\}$ for some factor $b'$ of the denominator of $g$. But then $b\in \{a-1+l/k,a-1,a-2\}$. From the assumption on $a$ and the inequality $k\geq |l|$, we get $l/k=1$ and thus $a=u+s+c$. 

We proved that the denominator of $g$ can only have factors of the form  $u+s+c$. Let $a'$ be the factor with maximal $\Re c$. Then the first term of the RHS has $Ua'=a'+1$ in the denominator, and no other term of \eqref{=3pq} has this factor in the denominator - contradiction. 

\DimaI{
For the uniqueness statement, suppose we have 2 tuples $(p_i,q_i,w_i,g_i), \, i\in \{1,2\}$ such that for each $i$, $p_i,q_i$ are monic, $g_i(u,s)$ is coprime to $p_i(s-u)$, to  $q_i(s+u)$ and to $w_i(s)$ and has no factors independent of $s$, and 
\begin{equation}
\frac{q_1(s+u)p_1(s-u)Sg_1}{w_1(s)g_1}=\frac{q_2(s+u)p_2(s-u)Sg_2}{w_2(s)g_2}
\end{equation}
We claim that $p_1=p_2$, $q_1=q_2$, $w_1=w_2$, and $g_1/g_2$ is a constant.
We prove this by induction on $d:=\min(\deg g_1,\deg g_2)$. The base case $d=0$ is straightforward.  For the induction step, assume WLOG $\deg g_1\geq \deg g_2$ and let $a$ be an irreducible factor in $g_1$. 
Since $g$ is coprime to $p(s-u), q(s+u)$, $a$ appears in the denominator of the reduced expression for the LHS, and thus also of the reduced expression for the RHS. Thus $a$ divides either $g_2(u,s)$ or $w(s)$. 
If $a$ divides $g_2$ then we can cancel it out from both $g-1$ and $g_2$, and the statement holds by induction.
Assume now that $a$ divides $w(s)$. Then $a$ depends only on $s$, and so does $Sa$, which is a factor in the numerator of the reduced expression, since $Sg_1$ is coprime to $w_1(s)$. Thus $Sa$ divides $Sg_2$ and thus $a$ divides $g_2$,  and the statement again follows by induction. }
\end{proof}

In the next section we find all the possibilities for the degrees of $R$, $w$, $p$, and $q$.

\section{Classification of $R$, $w$, $p$, and $q$}\label{sec:cases}

We can write 
\begin{equation}
\sigma=\delta_1/\gam, \quad \alp=\delta_0/\gamma, \quad U^{-1}(\beta/\sigma)=\delta_{-1}/\gamma,
\end{equation}
where $\gamma,\delta_1,\delta_0,\delta_{-1}\in \C[u]$ are polynomials that have no overall common factor. By \eqref{=f}, equation \eqref{=3tfh} becomes
\begin{multline}\label{=3pq2}
\gamma(u)w(s-1)S^{-1}(g)=\delta_1(u)q(s+u)q(s+u-1)Ug \,+\\ +(\delta_0(u)+\gamma(u)R(s))q(s+u-1)p(s-u-1)g+\delta_{-1}(u)p(s-u)p(s-u-1)U^{-1}g.
\end{multline}

We will show below that there are only \emph{finitely} many possibilities for degrees of $p,q,w,R$. In particular we will prove that these degrees are at most $4$. 
Denote 
$$d_1:=\deg p, \; d_2:=\deg q,\; d_3:=\deg w, d_4:=\deg R$$

\begin{prop}[See \S \ref{subsec:Jac} for the proof]\label{prop:Jac}
If \eqref{=3pq2} holds with $R=0$ then $d_3\leq 2$ and either $d_1, d_2\leq 1$ or $d_1 = d_2=d_3=2$.
\end{prop}

\begin{prop}[See \S \ref{subsec:Wil} for the proof]\label{prop:cases}
If \eqref{=3pq2} holds with $d_1,d_3,d_4>0$ then $d_1=1$ and one of the following holds
\begin{enumerate}[(a)]
\item $d_2=1,d_3=4,d_4=2$ \label{it:lWil}
\item $d_2=0,d_3=3,d_4=2$ \label{it:lcdH}
\item $d_2=1,d_3=3,d_4=1$ \label{it:lcH}
\item $d_2=0,d_3=2,d_4=1$\label{it:lMex}
\item $d_2=0,d_3=d_4=1$ \label{it:lChar}
\end{enumerate}
\end{prop}

This section is devoted to the proof of these propositions.
The key tool is the following lemma.  
\begin{lem}\label{lem:d_i-equation} If \eqref{=3pq2} holds  then the five polynomials 
$$a_1(x)=(x-1)^{2d_1}, \; a_2(x)=(x-1)^{d_1}(x+1)^{d_2}, \;a_3(x)=(x+1)^{2d_2}, \; a_4(x)=x^{d_3},\; a_5(x)=x^{d_4}(x-1)^{d_1}(x+1)^{d_2}$$ are linearly dependent over $\C$. Moreover, if $R\equiv 0$ then $a_1,\dots,a_4$ are linearly dependent.
\end{lem}

For the proof of this lemma  we need the following one.
\begin{lemma}\label{lem:GenPol}
Let $\cR$ be a commutative $\C$-algebra which is also an integral domain, and let $A_1,\dots,A_n,B_1,\ldots,B_n\in \cR[u]$ be polynomials.
Suppose that $B_1,\ldots,B_n$ have the same leading coefficient $b$. Suppose that $A_1B_1,\ldots ,A_nB_n$ are linearly dependent.
Let $a_1,\dots,a_n$ denote the leading coefficients of $A_n$. 
 Then $a_1,\dots,a_n$ are linearly dependent over $\C$. 
\end{lemma}
\begin{proof}
Let $d$ denote the maximum of degrees of the polynomials $A_iB_i$. Suppose without loss of generality that the polynomials $A_1B_1,\dots,A_kB_k$ have degree $d$, and the polynomials $A_{k+1}B_{k+1},\dots,A_nB_n$ have smaller degrees. Let $\sum_{i=1}^n c_iA_iB_i=0$ be a non-trivial linear relation, that exists by the condition. Inspecting the coefficients of degree $u^{d}$ in this relation, we obtain $(\sum_{i=1}^k c_ia_i)b=0$. Since $\cR$ is an integral domain, this implies $\sum_{i=1}^k c_ia_i=0$. 
\end{proof}

\begin{proof}[Proof of Lemma \ref{lem:d_i-equation}] 
Substituting $s=xu$ in \eqref{=3pq2} we obtain the identity
\begin{equation}
\label{=Ae}
A_1 g(xu,u+1)+A_2g(xu,u)+A_3g(xu,u-1)+A_4g(xu-1,u)+A_5g(xu,u) =0
\end{equation} where
\begin{align*}
A_1=\delta_{-1}(u)p(xu-u)p(xu-u-1),  \quad &A_2= \delta_0(u)p(xu-u-1) q(xu+u-1),\\A_3=\delta_1(u)q(xu+u)q(xu+u-1), \quad &A_4= -\gamma(u)w(xu-1), \\ A_5=R(xu)\gamma(u)p(xu-u-1) q(xu+u-1).
\end{align*}

Let $\cR:=\C[x]$ and view $A_i$ as polynomials in $\cR[u]$. It is easily seen that for each $i$, the leading coefficient of $A_i$ is a non-zero $\C$-multiple of $a_i(x)$, while each of the four $g$-factors \eqref{=Ae} has the \emph{same} leading $u$-coefficient. Thus the result follows from Lemma \ref{lem:GenPol}.
\end{proof}

\subsection{Proof of Proposition \ref{prop:Jac} and an extra result}\label{subsec:Jac}
In this subsection we assume that \eqref{=3pq2} holds with $R\equiv 0$. 
By Lemma \ref{lem:d_i-equation}, this implies that 
 there exist constants $c_1,c_2,c_3,c_4\in \C$, not all $0$, such that
\begin{equation}\label{=c}
c_1(x-1)^{2d_1}+c_2(x-1)^{d_1}(x+1)^{d_2}+c_3(x+1)^{2d_2}+c_4x^{d_3}\equiv0.
\end{equation}

\begin{lem}\label{lem:1stBound}
\DimaI{We have $|d_1-d_2|\leq 1$.} 
\end{lem}
\begin{proof}
Replacing $x$ by $-x$ in \eqref{=c} we can assume $d_1\geq d_2$. Thus we have to show $d_1\leq d_2+1$.

Suppose by way of contradiction that $d_1\geq d_2+2$. Then the polynomial $(x-1)^{2d_1}$ has, among others, monomials of degrees $2d_1$ and $2d_1-1$. No term of $(x+1)^{2d_2}$ or of $(x-1)^{d_1}(x+1)^{d_2}$  can cancel any of these, and the term $x^{d_3}$ can only cancel one of them. Thus we must have $c_1=0$. But then, the polynomial $(x-1)^{d_1}(x+1)^{d_2}$ has terms of degrees $d_1+d_2$ and $d_1+d_2-1$ and again $x^{d_3}$ can only cancel one of them. Thus we have $c_1=c_2=0$ and thus $c_3,c_4\neq 0$. If $d_2>0$ then $(x+1)^{2d_2}$ has at least three terms, and they cannot be cancelled by $c_4x^{d_3}$. Thus we must have $d_2=d_3=0$. Then looking at the degree of $s$ in \eqref{=3pq2}, we see that the term $\delta_{-1}p(s-u)p(s-u-1)U^{-1}g$ has higher degree in $s$ then the other terms. This is a contradiction. 
\end{proof}
\begin{lem}\label{lem:2ndBound}
Suppose that $d_1,d_2\geq 1$. Then $d_1=d_2\in \{1,2\}$.
\end{lem}
\begin{proof}
Suppose $c_4=0$ in \eqref{=c}. Then setting $x=\pm 1$ we deduce $c_1=c_3=0$ and hence $c_2=0$ also. Thus we may assume $c_4\neq 0$. Now setting $x=\pm 1$ we get $c_1,c_3\neq 0$.

Let $\lambda$ be a root of the quadratic polynomial $c_1t^2+c_2t+c_3$, and let $\mu$ be a root of the polynomial $(x-1)^{d_1}-\lambda(x+1)^{d_2}$. Then setting $x=\mu$ in \eqref{=c} the sum of the first three terms vanish and we get $c_4\mu^{d_3}=0$, which implies $\mu=0$ (and $d_3\geq 1$).
Thus 0 is the only root of $(x+1)^{d_1}-\lambda (x-1)^{d_2}$, which means this polynomial is a monomial:
\begin{equation}\label{=k}
(x-1)^{d_1}-\lambda (x+1)^{d_2}=\eps x^{d_5}\quad \text{ for some }\eps,d_5
\end{equation}

 If $d_5=0$ then $d_1=d_2=1$, so we may assume $d_5>0$. 
Setting $x=0, 1,-1$ in \eqref{=k} we get $$\lambda=(-1)^{d_1}, \, \eps=(-1)^{d_1+1}2^{d_2}, \, \eps=(-1)^{d_1+d_5} 2^{d_1}.$$
This forces $d_1=d_2=:d$, and \eqref{=k} becomes
\begin{equation}
(x-1)^d-(-1)^{d}(x+1)^d=(-1)^{d+d_5} 2^dx^{d_5}
\end{equation}

The coefficient of $x$ on the left is $2d(-1)^{d-1}$ which implies $d_5=1$ and $2d=2^d$. Thus $d\in \{1,2\}$.
\end{proof}

\begin{lem}\label{lem:Maple}
 \DimaI{If $d_1=d_2=2$ then $d_3 = 2$} and if in addition $g=1$ then $\delta_1+\delta_0+\delta_{-1}=0$.
\end{lem}
\begin{proof}
Substituting $d_1=d_2=2$ in \eqref{=c} we obtain
\begin{equation}
-c_4x^{d_3}=(c_1+c_2+c_3)x^4+4(c_3-c_1)x^3+2(3c_1-c_2+3c_3)x^2+4(c_3-c_1)x+(c_1+c_2+c_3)
\end{equation}
\DimaI{Since the LHS is a monomial, so is the RHS. Since the coefficients $4(c_3-c_1)$ and $c_1+c_2+c_3$ appear twice, both must vanish. Thus, $c_3=c_1$,  $c_2=-2c_1$, and $d_3=2$.
Thus the top degree terms in $s$ of the RHS of \eqref{=3pq2} must cancel, and thus if  $g=1$ then $\delta_1+\delta_0+\delta_{-1}=0$. }
\end{proof}

\begin{proof}[Proof of Proposition \ref{prop:Jac}.]
If $d_1,d_2\leq 1$ then the degree in $s$ of the RHS of \eqref{=3pq2} is at most $\deg g+2$. Thus $d_3\leq 2$.
Suppose now $\max(d_1,d_2)\geq 2$. Then, by Lemma \ref{lem:1stBound}, $\min(d_1,d_2)\geq 1$, and Lemma \ref{lem:2ndBound} implies $d_1=d_2=2$.  By Lemma \ref{lem:Maple} this implies $d_3\leq 2$. 

Altogether,  we obtain $d_3\leq 2$, and either $d_1,d_2\leq 1$, or $d_1=d_2=2$.
\end{proof}


\DimaI{In the special case $g=1$ we can further restrict the options for $p,q,$ and $w$.}
\begin{lem}\label{lem:NewFpq}
\DimaI{If $d_1=d_2=2$,} $w(-1)=0$, $g=1$, $R=0$, and $w$ is monic then one of the following holds.
\begin{enumerate}[(a)]
\item $w(s)=(s+1/2)(s+1)$, $p(t)=q(-1-t)$ \label{it:1pq}
\item $w(s)=(s+1)(s+3/2)$, $p(t)=q(-2-t)$ \label{it:2pq}
\end{enumerate}
\end{lem}
\begin{proof}
Shifting $u$ and $s$ we can assume $p(t)=t^2+c$, $q(t)=t^2+d$ for some $c,d$.
By Lemma \ref{lem:Maple} we have $d_3 = 2$ and $\delta_1+\delta_0+\delta_{-1}=0$.
 Since $\deg w=d_3= 2$, the term of $s^3$ in the RHS of \eqref{=3pq2} must vanish. But this term is $4(\delta_1-\delta_{-1})u-2(\delta_1+2\delta_0+\delta_{-1})$.
Since $\delta_1+\delta_0+\delta_{-1}=0$, we obtain that $\delta_0=2u(\delta_1-\delta_{-1})$, and thus $(2u+1)\delta_{1}=(2u-1)\delta_{-1}$. 
Note that any common factor of $\delta_{-1}$ and $\delta_1$ divides also $\delta_0$, and thus also $\gamma$ (since $R=0$). 
Assuming without loss of generality that the leading coefficient of $\delta_{-1}$ is 2, we obtain that $\delta_{-1}=2u+1$ and $\delta_1=2u-1$. This implies $\delta_0=-4u$.

Then the coefficient of $s^2$ in the RHS of \eqref{=3pq2} is $32u^3-8u+2(d-c)$.
 Since $w$ is monic, we have $\gamma(u)=32u^3-8u+2(c-d)$.
The coefficient of $s$ in the RHS of \eqref{=3pq2} is $2(d -c )+ 12u - 16u^2(c-d)  - 48u^3$. Since this has to be divisible by $\gamma$, we have $c=d$. In this case $w(s-1)=(s-1)(s-1/2)$, $\gamma=8u(4u^2-1)$.
Thus $w(s)=s(s+1/2)$.

Shifting back in $s$ we get the same shift for $p$ and $q$, while shifting in $u$ we obtain opposite shifts.
Since $w(-1)=0$ after the shift, we can shift $s$ either by $1/2$ or by $1$. If we shift by $1/2$ we get case \eqref{it:1pq}, and if we shift it by $1$ we get case \eqref{it:2pq}.
\end{proof}
This lemma, together with Proposition \ref{prop:Jac}, implies the following corollary.
\begin{cor}\label{cor:pqDet}
Suppose that \eqref{=3pq2} holds with $R\equiv 0$, $g=1$, $w(-1)=0$, and $w$ is monic. Then one of the following holds.
\begin{enumerate}[(a)]
\item $\deg p,\deg q\leq 1$
\item $w(s)=(s+1/2)(s+1)$, $p(t)=q(-1-t)$
\item $w(s)=(s+1)(s+3/2)$, $p(t)=q(-2-t)$
\end{enumerate}
\end{cor}
\subsection{Proof of Proposition \ref{prop:cases}}\label{subsec:Wil}

\begin{lem}\label{lem:ci}
If \eqref{=3pq2} holds with $d_4\geq 1$ then there exist constants $c_1,c_2,c_3,c_4\in \C$ such that
\begin{multline}\label{=c2}
c_1(x-1)^{2d_1}+c_2(x-1)^{d_1}(x+1)^{d_2}+c_3(x+1)^{2d_2}+c_4x^{d_3}+x^{d_4}(x-1)^{d_1}(x+1)^{d_2}\equiv0,\\
\text{ and  if }d_3\neq d_1+d_2+d_4 \text{ then } c_4=0
\end{multline}
\end{lem}
\begin{proof}
We have to show that in \eqref{=Ae},
the factor $A_5$ participates in the leading $u$-coefficient. Suppose the contrary. By the previous subsection, in this case we have $d_3\leq 2$. Since $d_1\geq 1, $  this implies that if $c_4\neq 0$ then $d_3\leq d_1+d_2+1$. Since $d_4\geq 1$, this in turn implies that if $A_5$ does not participate in leading $u$-coefficient then neither does $A_4$ , since the degree of $u$ in $A_5$ is $\deg \gamma+d_4+d_1+d_2$, and in $A_4$ it is $\deg \gamma+d_3$. Thus $c_4=0$, and  the set $\{(x-1)^{2d_1},(x-1)^{d_1}(x+1)^{d_2},(x+1)^{2d_2}\}$ is linearly dependent. This is impossible, since $d_1\geq 1$. 

Finally, if $d_3\neq d_1+d_2+d_4$ then $\deg \gamma+d_3\neq \deg \gamma+d_4+d_1+d_2$, and the degrees of $u$ in $A_5$ and in $A_4$ are different, and cannot both be maximal. Thus in this case $c_4=0$.
\end{proof}

\begin{lemma}\label{lem:<=1}
If \eqref{=c2} holds with $d_3,d_4>0$ then  $d_1,d_2\leq 1$.
\end{lemma}
\begin{proof}
It is enough to show that $d_1\leq 1$, since the other statement follows by the coordinate change $x\mapsto -x$. Assume the contrary: $d_1\geq 2$. Then $c_3(x+1)^{2d_2}+c_4x^{d_3}$ is divisible by $(x-1)^{d_1}$. 

Suppose first that $c_3=c_4=0$. Then dividing by $(x-1)^{d_1}$ we obtain
\begin{equation}\label{=c0}
c_1(x-1)^{d_1}+c_2(x+1)^{d_2}+x^{d_4}(x+1)^{d_2}\equiv0.
\end{equation}
If $d_2>0$ we substitute $x=-1$ and obtain $c_1=0$, which is impossible since $(x+1)^{d_2}$ is not divisible by $x$. If $d_2=0$ then $c_1(x-1)^{d_1}+c_2+x^{d_4}\equiv0$.
This implies $d_1=1$. 

Thus we can assume $(c_3,c_4)\neq (0,0)$. 
Shifting $x$ by $1$ for convenience we obtain that $x^{d_1}$ divides
$c_3(x+2)^{2d_2}+c_4(x+1)^{d_3}.$ Since we assumed $d_1\geq 2,$ this implies the vanishing of the constant term and the coefficient of $x$. These coefficients are
$2^{2d_2}c_3+c_4$ and $c_3d_22^{2d_2}+c_4d_3$. Thus $d_2=d_3$. 

If $d_1\geq 3$ we also have $\binom{2d_2}{2}2^{2d_2-2}=2^{2d_2}\binom{d_3}{2}$, which contradicts $d_2=d_3$. Thus we have $d_1\leq 2$. By a symmetric argument, this implies $d_2\leq 2$ as well.

It is left to consider the case $d_1=2, d_2=d_3\in \{1,2\}, d_4\geq 1$. 
From the term $x^{d_4+d_1+d_2}$ in \eqref{=c2} we see that we must have $d_2=d_3=d_4=1$ and $c_1=-1$. Substituting $x=0$ we get $c_1+c_2+c_3=0$, and substituting $x=-1$ we get $c_4=-16c_1$. However, since in this case $d_3=1<d_1+d_2+d_4=4$, the condition in \eqref{=c2} says $c_4=0$, contradicting $c_4=-16c_1=16$. Thus this case is impossible.
\end{proof}

\begin{lem}\label{lem:>0}
If \eqref{=c2} holds with  $d_1,d_2,d_3,d_4>0$ then $d_1=d_2=1, d_3=d_4+2,$ and $d_4\in \{1,2\}$.
\end{lem}
\begin{proof}
By the previous lemma we have $d_1=d_2=1$. Thus \eqref{=c2} has a term $x^{d_4}(x^2-1)$. To cancel this term we must have $d_3=d_4+2$ and $d_4\leq 2$, since other terms have degrees $d_3,2,1,0$. 
\end{proof}

\begin{lem}\label{lem:d20}
If \eqref{=c2} holds with  $d_1,d_3,d_4>0=d_2$ then $d_1=1$ and either $(d_3,d_4)= (3,2)$ or $d_4=1$.
\end{lem}
\begin{proof}
By Lemma \ref{lem:<=1} we have $d_1=1$. Thus, by \eqref{=c2}, the polynomial $c_4x^{d_3}+x^{d_4}(x-1)$ has degree at most 2. This implies that either $d_3=d_4+1$ and $d_4\leq 2$, or $d_4=1$. 
\end{proof}




\begin{proof}[Proof of Proposition \ref{prop:cases}]
By Lemma \ref{lem:ci}, \eqref{=c2} holds.
By Lemma \ref{lem:<=1}, we have $d_1=1$ and $d_2\leq 1$. If $d_2=1$, then by Lemma \ref{lem:>0} either \eqref{it:lWil} or \eqref{it:lcH} holds.

If $d_2=0$ then by Lemma \ref{lem:d20}, either $(d_3,d_4)= (3,2)$ or $d_4=1$.
If $(d_3,d_4)= (3,2)$ (and $d_2=0$) then \eqref{it:lcdH} holds.
Otherwise $d_4=1$ and the top degree of $s$ in the right-hand side of \eqref{=3pq2} is $\leq 2+\deg g$, and in the left-hand side $d_3+\deg g$. Thus $d_3\leq 2$ and either \eqref{it:lMex} or \eqref{it:lChar} holds. 
\end{proof}

\section{Proof of Theorem \ref{thm:rat}}\label{sec:PfRat}

Let $M$ be the $A$-module defined by $\cP$ in \S \ref{sec:PolMod}. By Theorem \ref{thm:1dim}, it is one-dimensional.
By abuse of notation, we will denote the element of $M$ defined by $\cP$ also by $\cP$. Let $f,h\in \C(u,s)$ be such that $S\cP=f\cP$ and $U\cP=h\cP$. By Theorems \ref{thm:mod} and \ref{thm:3termForm} we have
\begin{equation}\label{=fg}
f=\frac{q(u+s)p(s-u)Sg}{w(s)g}
\end{equation}
where  $g$ is coprime to $p(s-u), q(s+u)$ and $w(s-1)$, and $p,q$ are monic.


By Lemma \ref{lem:fxyRat} we have  $p(0)=0$ and $w(-1)=0$.
Thus $\deg p,\deg w\geq 1$ and $p(t)=tp_1$, $w=(s+1)w_1$ for some polynomials $p_1,w_1$. Let us check how change of the variable $z$ of the polynomials effects $f$\ and $R$. If we rescale $z\mapsto \lambda z$, and then renormalize both $\Phi_k$ and $P_n$ the family to make them monic again, the new family will be given by $(\lambda^{-1}R,\lambda f)$. If we shift $z\mapsto z+\mu$, the new family will be given by $(R+\mu,f)$. Thus we can assume that the constant term of $R$ is zero, and $R$ is either zero or monic. Thus $R(s)=sr(s-1)$ for some $r\in \C[s]$ that is either zero or monic. By Propositions \ref{prop:Jac} and \ref{prop:cases}, $\deg r\leq 1$. Consider the following cases.

\begin{enumerate}[{Case} 1]
\item  $r\equiv 0$.

 By Proposition  \ref{prop:Jac} we have $\deg p, \deg w\leq 2$. Using the rescaling $z\mapsto \lam z, f\mapsto \lam f$, we can assume that $p,q,$ and $w$ are monic.  

Consider first the case $\deg p=2$. By Proposition  \ref{prop:Jac} in this case we have $\deg q=\deg w=2$.
Then $\deg w_1=\deg p_1=1$ and thus $w_1=s+b, \, p_1=t+d$ for some $b,d\in \C$. 
  This is case \eqref{it:rat:b}.

It is left to consider the case $\deg p=1$, {\it i.e.} $p_1=1$. By Proposition  \ref{prop:Jac},  in this case $\deg q\leq 1$, $\deg w_1\leq1$.
Thus in order to show that we are in case \eqref{it:rat:a} it is left to exclude the possibility $\deg q=0, \, \deg w_1=0$, i.e. $w_1=1$, and $q=1$. If this was the case, the degree in $s$ of the LHS of \eqref{=3pq2} would be $\deg_sg+1$, while the degree in $s$ of the RHS would be $\deg_s g+ 2$. This is a contradiction, and thus this option is not possible. 

\item $r\not \equiv 0$.

In this case   by Proposition \ref{prop:cases} we have $p_1\equiv 1$, $\deg q, \deg r\leq 1$, and either\\ $\deg w_1=\deg q+\deg r+1$ or $q\equiv1\equiv w_1\equiv r$.

Thus it is left to prove that $w_1(s)=sq(s)r(s)+v(s)$ with $\deg v\leq 1$. In other words, we have to show that 
\begin{equation}\label{=w1}
\deg(w_1(s)-sq(s)r(s))\leq 1
\end{equation}
If $q\equiv1\equiv w_1\equiv r$ this clearly holds. If $\deg w_1\leq 1$, \eqref{=w1} also holds automatically. Suppose now that $\deg w_1\in \{2,3\}$, i.e. $\deg q+\deg r\in \{1,2\}$.

Substituting $p(t)=t$, $w(s)=(s+1)w_1(s),$ and $R(s)=sr(s-1)$ into \eqref{=3pq2} we obtain
\begin{multline}\label{=3pqB}
\gamma(u)sw_{1}(s-1)S^{-1}(g)=\delta_1(u)q(s+u)q(s+u-1)Ug \,+\\ +(\delta_0(u)+\gamma(u)sr(s-1))q(s+u-1)(s-u-1)g+\delta_{-1}(u)(s-u)(s-u-1)U^{-1}g.
\end{multline}

Consider $g$ as a polynomial in $s$ with coefficients in $\C[u]$, and denote its degree by $d$. 
Since $\deg q\leq 1,$ the only terms in \eqref{=3pqB} that can have $s$-degree more than $d+2$ are $$\gamma(u)sw_{1}(s-1)S^{-1}(g) \text{ and }\gamma(u)sr(s-1)q(s+u-1)(s-u-1)g.$$ Thus denoting by $\deg_s$ the degree in $s$, \eqref{=3pqB} implies 
\begin{equation}
\deg_s(w_{1}(s-1)S^{-1}(g)-r(s-1)q(s+u-1)(s-u-1)g)\leq d+1
\end{equation}
Shifting $s\mapsto s+1$ we obtain
\begin{equation}\label{=3pq2B}
\deg_s(w_{1}(s)g-r(s)q(s+u)(s-u)Sg)\leq d+1
\end{equation}

This implies that $w_1(s)$ is monic. If $\deg w_1=2$, the fact that $w_1(s)$ is monic  implies \eqref{=w1}. Suppose now that $\deg w_1=3$. Then $\deg r=\deg q=1$, thus $r(s)=s+a$ and $q_1(t)=t+c$ for some $a,c\in \C$. Write $w_1(s)=s^3+xs^2+ys+b$. Then \eqref{=w1} is equivalent to the equality $x=a+c$. 
Rewrite \eqref{=3pq2B} as 
\begin{equation}\label{=3pq2C}
\deg_s(w_{1}(s)(g-Sg)+(w_1(s)-r(s)q(s+u)(s-u))Sg)\leq d+1
\end{equation}

Denote by $LT$ the leading term in $s$, and let  $z(u):=LT(g)$. Then from \eqref{=3pq2C} we have $$LT(w_1(s)(g-Sg))=LT((w_1(s)-r(s)q(s+u)(s-u))Sg).$$ Now, 
$$LT(w_1(s)(g-Sg))=-dz(u)\text{ and }LT((w_1(s)-r(s)q(s+u)(s-u))Sg)=(x-a-c)z(u).$$
Thus $x=a+c-d$. Let $\lam$ be a root of $w_1(s-1)$. 
Replacing $g(u,s)$ by $g'(u,s):=\frac{g(u,s)}{(s+\lam)_{(d)}}$, we 
replace $w_1$ by $w_2$ satisfying $w_2(s-1)=s^3+(x+d)s^2+y's+b'$, and thus satisfying \eqref{=w1}.

\end{enumerate}
\proofend

\section{Proof of Theorem \ref{thm:main}, Proposition \ref{th:JacAlp}, and Corollary \ref{cor:ab}}\label{sec:PfMain}
\subsection{Proof of Theorem \ref{thm:main}}
Let $\cP=\{P_n\}_{n=0}^{\infty}$ be an HG-family given by the pair $(R,f)$. Let $\{\Phi_k\}_{k=0}^{\infty}$ be the corresponding Newtonian basis, and $c(n,k)$ be the corresponding sequence of coefficients. 
By Theorem \ref{thm:rat} there a polynomial $g\in \C[u,s]$ such that the reduced form of $f$ is
\begin{equation}\label{=fA}
f=\frac{Sg}{g}\cdot \frac{s-u}{s+1}\cdot\frac{p_1(s-u)q(u+s)}{w_1(s)},
\end{equation}
where $p_1,q_1,w_1\in \C[t]$ are polynomials in one variable that belong to a very restricted list. In particular, $D(s)=g(u,s)(s+1)w_1(s)$. This implies that $g$ does not depend on $u$. Thus $Ug=U^{-1}g =g$. 

Let us show that $g$ is constant.
Suppose by way of contradiction that $g$ is not constant, and let $c$ be a root of $g$ that is not a root of $S^{-1}g$. Substitute $s:=c$ into the 3-term relation \eqref{=3pq2}. Then the right-hand side vanishes, while the left-hand side does not. This is a contradiction, thus $g$ must be constant, and thus we can assume $g=1$.

In case \eqref{it:rat:a} of Theorem \ref{thm:rat}, the pair $(R,f)$ satisfies the conditions of \eqref{it:W} of Theorem \ref{thm:main}.

In case \eqref{it:rat:b} of Theorem \ref{thm:rat}, $R$ is a constant and $$f(u,s)=\frac{s-u}{s+1}\cdot\frac{(s-u+1-c)q(u+s)}{s+d},$$
for some  quadratic polynomial $q\in \C[t]$ and scalars  $c,d\in \C$. Using a shift of variable $z\mapsto z+\mu$, we can assume that $R=0$. Using rescaling $z\mapsto \lambda z$, we can assume that $q$ is monic. 
Now we have to show that either $d=1/2$ and $q(t)=(t+1)(t+c)$, or $d=3/2$ and $q=(t+2)(t+c+1)$.  This follows from Corollary \ref{cor:pqDet}.
\proofend

\subsection{Proof of Proposition \ref{th:JacAlp}}
Let $\{P_n\}_{n=0}^{\infty}$ be an HG family given by a pair
 $(R,f)$. Let us remind the definition of the sequences $\alp_n$ and $\beta_n$  in the proposition.
Let $f_1$ and $f_2$ be as in \eqref{=f12}:
\begin{equation*}
f_1(u):=f(u+1,u)^{-1} \text{ and } f_2(u):=f(u+2,u)^{-1}f_1(u+1),
\end{equation*}
 and define  rational functions  in  $\alp$ and $\beta$ in one variable $u$ by
\begin{align*}
\alp(u)=f_1(u-1)-f_1(u)-R(u), \text{ and }
\beta(u)=f_2(u-1)-f_2(u)-f_1(u)(\alp(u+1)+R(u))
\end{align*}
The sequences $\alp_n$ and $\beta_n$ are given by  $\alpha_n=\alpha(n)$ and $\beta_n=\beta(n)$ for all $n>0$, and 
 $$\alpha_0 =-f_1(0)-R(0), \quad \beta_0=-(f_1(0)-f_1(1)-R(1)+R(0))f_1(0)-f_2(0).$$

Suppose that $(R,f)$ is as in one of the cases (\ref{it:W}-\ref{it:F}) of Theorem \ref{thm:main}. 
We have to show that the family $\{P_n\}_{n=0}^{\infty}$ satisfies the recursion \eqref{=3t}: 
\begin{equation}\label{=3t3}
zP_n=P_{n+1}+\alp_nP_n+\beta_{n-1}{P_{n-1}}  \,\,\forall  n\in \Z_{\ge 0}
\end{equation}

For $n=0$ this is an easy verification.  For $n>0$ we have to show that 
\begin{equation}\label{=3tc3}
c(n,k-1)=c(n+1,k)+(\alp_n+R(k)) c(n,k)+\beta_{n-1} c(n-1,k)
\end{equation}
We have $$c(n,k)=f(n,k-1)c(n,k-1) \text{ and }c(n+1,k)=h(n,k)c(n,k)$$ for some $h(u,s)\in \C(u,s)$. We know that 
 \begin{equation}\label{=f3}
f=\frac{q(s+u)p(s-u)}{w(s)}, \quad h= \frac{\sigma(u)q(s+u)}{p(s-u-1)}.
\end{equation}
Substituting $k:=n$ we have 
$$h(n,n)= \frac{c(n+1,n)}{c(n,n)}=f^{-1}(n+1,n)=f_1(n)$$
and 
$$h(n,n)= \sigma(n)\frac{q(2n)}{p(-1)}$$
Thus $\sigma(n)=f_1(n){p(-1)}/{q(2n)}$.
Now, the recursion \eqref{=3tc3} is equivalent to the identity 
\begin{multline}
w(s-1)=\sigma(u)q(s+u)q(s+u-1) \,+\\ +(\alp(u)+R(s))q(s+u-1)p(s-u-1)+(\beta(u-1)/\sigma(u-1))p(s-u)p(s-u-1).
\end{multline}
Substituting $f_1(u)p(-1)/q(2u)$ for $\sigma(u)$ the identity is equivalent to
\begin{multline}
{q(2u)f_1(u-1)p(-1)}w(s-1)=f_1(u){p(-1)}f_1(u-1)p(-1)q(s+u)q(s+u-1) \,+\\ +(\alp(u)+R(s)){q(2u)f_1(u-1)p(-1)}q(s+u-1)p(s-u-1)+\beta(u-1)q(2u-2){q(2u)}p(s-u)p(s-u-1).
\end{multline}
Verification of this identity in each of the cases in Theorem \ref{thm:main} is long but straightforward.
\proofend

\subsection{Proof of Corollary \ref{cor:ab}}
Let $N(u,s)$ and $D(s)$ be the numerator and the denominator of $f$ in the reduced form.
\begin{lem}[Appendix \ref{subsec:Jacrat}]\label{lem:fxy}Suppose that $\cP$ is quasi-orthogonal. Then
\begin{enumerate}[(i)]
\item \label{it:fxy} $D(k)\neq 0$ for any $k\in \Z_{\geq 0}$. 
\item \label{it:xy!0} For all $n\geq k \geq 0$ we have $c(n,k)\neq 0$, and for all $n> k \geq 0$ we have $N(n,k)\neq 0$.
\item \label{it:xyc} For every $n\geq k\geq 0$ we have $c(n,k)=\prod_{i=k}^{n-1}f(n,i)^{-1}$.
\end{enumerate}
\end{lem}

\begin{proof}[Proof of Corollary \ref{cor:ab}]
\begin{enumerate}[(i)]
\item If $\cP$ is well-defined and quasi-orthogonal, then $f_1(n)$ and $f_2(n)$ are finite by Lemma \ref{lem:fxy}. By Proposition \ref{th:JacAlp}, it satisfies the recursion \eqref{=3t3} with $\alp_n,\beta_n$ defined in Proposition \ref{th:JacAlp}. By the Gauss-Favard theorem (Theorem \ref{thm:Fav}) we must have $\beta_n\neq 0$.

Conversely, if $f_1(n)$ and $f_2(n)$ are finite then, by Proposition \ref{th:JacAlp} the family satisfies the recursion \eqref{=3t3} with $\alp_n,\beta_n$ defined in Proposition \ref{th:JacAlp}. If in addition $\beta_n\neq 0$ for all integer $n\geq 0$ then by Theorem \ref{thm:Fav} the family is quasi-orthogonal.

\item Follows from the recursion \eqref{=3t3} and Theorem \ref{thm:Fav}.

\item Follows from the recursion \eqref{=3t3}.
\end{enumerate}
\end{proof}

\section{Proof of Theorem \ref{thm:Classmain}}\label{sec:PfClassmain}
We start with the following lemma, that follows from the definition of hypergeometric functions (see \eqref{=pFq} in \S \ref{subsec:class}) by direct computation. 
\begin{lem}[Direct computation]\label{=resf}
Let
\begin{eqnarray}
&Q^0_n(z):=\frac{(-1)^n\prod_{k=1}^j(\delta_k)_{(n)}}{e^n\prod_{l=1}^i(\gamma_l)_{(n)}}\cdot {\,}_i F_j\left(\begin{matrix}-n& &\gam\\&\delta
\end{matrix};ez\right),\\
&Q^1_n(z):=\frac{(-1)^n\prod_{k=1}^j(\delta_k)_{(n)}}{e^n\prod_{l=1}^i(\gamma_l)_{(n)}}\cdot {\,}_i F_j\left(\begin{matrix}-n& &\gam & z  \\&\delta
\end{matrix};e\right)\\
&Q^2_n(-z(z-a)):=\frac{(-1)^n\prod_{k=1}^j(\delta_k)_{(n)}}{e^n\prod_{l=1}^i(\gamma_l)_{(n)}}\cdot {\,}_i F_j\left(\begin{matrix}-n& &\gam & z & -z+a\\&\delta
\end{matrix};e\right)
\end{eqnarray}
Then $\{Q^i_n\}$ are monic HG-families with  
 \begin{equation}\label{=resf2}
f(u,s)=e\frac{s-u}{s+1}\frac{\prod_{l=1}^i(s+\gamma_l)}{\prod_{k=1}^j(s+\delta_k)}
\end{equation}
and 
\begin{equation}\label{=Ri}
R(s)=\begin{cases}
                        0, & i=0\\
            s, & i=1\\
            s(s+a), & i=2
                 \end{cases}
\end{equation}

Conversely, if  $\{P_n\}$ is a monic polynomial HG-family with  $f$ given by \eqref{=resf2} and $R\in \{0,s,s(s+a)\}$ then $P_n(z)=Q^i_n(z)$ where $i$ is defined by $R$ via \eqref{=Ri}.
\end{lem}

Using the coordinate shift $z\mapsto z+\mu$ we can add a constant to $R$ and assume $R$\ has no constant term. Furthemore, using rescaling $z\mapsto \lambda z$  we multiply $f$ by $\lambda$ and divide $R$ by $\lambda$.  Thus from now on we assume  $R(s)=sr(s-1)$ with $r$ monic or zero, where the shift $s-1$ is for convenience.

Thus to obtain the polynomials in Theorem \ref{thm:main} we take $Q^i_n(z)$ with $i$ corresponding to $R$ as in \eqref{=Ri}, with $\gamma$ 
empty if $q=1$ in case \eqref{it:W}, $\gamma=\{n+c\}$ if $q(t)=t+c$  in case \eqref{it:W}, $\gamma=(1-n-c, n+1,n+c)$ for $E_n^{(c)}$ (case \eqref{it:E}) and $\gamma=(1-n-c, n+2,n+c+1)$ for $F_n^{(c)}$ (case \eqref{it:F}). The $\delta$ is the set of roots of $D(s)/(s+1)=w(s)/(s+1)$.

Using Lemma \ref{=resf}, we see that cases \eqref{it:E} and \eqref{it:F} of Theorem \ref{thm:main} correspond to cases \eqref{it:AE} and \eqref{it:AF} of Theorem \ref{thm:Classmain}.
The table in \eqref{tab:decomp} describes how case \eqref{it:W} of Theorem \ref{thm:main} decomposes to cases (\ref{it:AJac}-\ref{it:ACha})  of Theorem \ref{thm:Classmain}, depending on the degrees of $R,q$, and $v$.

To verify the areas of definition specified in Theorem \ref{thm:Classmain}, we need formulas for $\alp$ and $\beta$ for each case. The formulas are given by Proposition \ref{th:JacAlp}. For the classical families given in  cases (\ref{it:AJac}-\ref{it:ACha})  of Theorem \ref{thm:Classmain}, they coincide with the classically known formulas, see {\it e.g.} \cite[Ch. 9]{KLS}. 
For the family $E_n$ (case \eqref{it:AE} of Theorem \ref{thm:Classmain}) they are given by 
\begin{equation}
\alp(u)=-\frac{1}{2(2u+c-1)(2u+c+1)}, \quad \beta(u)=\frac{1}{16(2u+c-2)(2u+c-1)^2(2u+c)} \end{equation}
For the family $F_n$ (case \eqref{it:AF} of Theorem \ref{thm:Classmain}) they are given by 
\begin{equation}
\alp(u)=-\frac{1}{2(2u+c)(2u+c+2)}, \quad \beta(u)=\frac{1}{16(2u+c-1)(2u+c)^2(2u+c+1)} \end{equation}

Also, for both families we have $\beta_0=\beta(0)$, but for $F_n$ we have $\alp_0=\alp(0)$ and for $E_n$ we have $\alp_0=-1/(4c(c+1))\neq -1/(2(c-1)(c+1))=\alp(0)$. 
This one value gives the difference between $E^{(c)}_n$ and $F^{(c-1)}_n$.

\subsection{Proof of Theorem \ref{thm:Classmain}}

First of all, it is easy to check that if the conditions are satisfied, then the corresponding families defined by hypergeometric functions as in Theorem \ref{thm:main} are well-defined.  Furthermore, they are quasi-orthogonal since they satisfy the 3-term relation, with $\alpha_n=\alp(n)$ well-defined, and $\beta_n=\beta(n)$ well-defined and non-zero for all $n\in \Z_{>0}$. Thus the conditions are sufficient. Let us now show that they are necessary.

Since $D(k)\neq 0$ for any $k\in \Z_{\geq 0}$, we have $b\notin \Z_{\leq 0}$ for the Jacobi and the Laguerre families. Since $N(n+1,n)\neq 0$ and $N(n+2,n)\neq 0$  for all $n\in \Z_{\geq 0}$, we have $a\notin \Z_{<0}$ for the Jacobi and the Bessel families, and $c\notin \Z_{\leq 0}$ for $E_n^{(c)}$ and  $F_n^{(c)}$.  

Let us now analyze the $\beta(n)$ and $\beta_0$. By Proposition \ref{th:JacAlp}, we have
\begin{equation}\label{=b00}
\beta_0=\beta(0)-\frac{w(0)w(-1)}{p(-1)p(-2)q(0)q(1)}
\end{equation}
In particular, $\beta_0=\beta(0)$ unless $p(-1)p(-2)q(0)q(1)=0$.

By Proposition \ref{th:JacAlp} we have for the Jacobi family:
$$\beta(n)=\frac{(u+1)(u+a)(u+b)(u+1+a-b)}{(2u+a)(2u+a+1)^2(2u+a+2)}$$
The requirement that $\beta(n)$ is non-zero and well-defined for all $n\in \Z_{>0}$ implies that in addition $a-b\notin \Z_{<-1}$ for the Jacobi family.
For $a\neq 0$ we have $\beta_0=\beta(0)$, and thus $a-b\neq -1$. For $a=0$ we have 
$$\beta_0=\beta(0)+\frac{w(0)w(-1)}{p(-1)p(-2)q(0)q(1)}=\beta(0)+\frac{(b+1)b(b-1)}{2}=\frac{b^2(b-1)}{2}$$
Thus $b\neq 1$, and $a-b\neq -1$. Altogether we have $a-b\notin \Z_{<0}$.

Thus, the conditions  on the parameters specified in Theorem \ref{thm:Classmain} must hold in cases (\ref{it:AJac}-\ref{it:ALag}) and (\ref{it:AE},\ref{it:AF}). Let us now treat the remaining 5 cases.

\begin{enumerate}[(a)]
\setcounter{enumi}{3}
\item Wilson: we have $D(k)=w(k)=(k+1)(k+a+b)(k+a+c)(k+a+d)$. Thus the condition $D(k)\neq 0 \,\forall k\in \Z_{\geq 0}$ implies 
$a+b,\,a+c,\,a+d\notin \Z_{\leq 0}.$ The condition $N(n,k)\neq 0$ for all $n>k\geq 0\in \Z$ implies $a+b+c+d\notin \Z_{\leq 0}.$ 

From the formula for $\beta$ in \cite[\S 9.1]{KLS} we see that $b+c,c+d,b+d\notin\Z_{<0}$. Suppose now that $0\in \{b+c,c+d,b+d\}$. Then $\beta(0)=0$, and thus $q(0)=0$. Thus $a+b+c+d=1$, and thus $1\in \{a+d,a+b,a+c\}$. This implies that $w$ has a double zero at $-1$, while $q$ has a simple zero at 0. Thus \eqref{=b00} implies $\beta_0=\beta(0)=0$. Thus we must have $b+c,c+d,b+d\notin\bZ_{\leq 0}$.

\item Continuous dual Hahn: 
by \cite[\S 9.3]{KLS} we have 
$$\beta(u)=-(u+a+b)(u+a+c)(u+1)(u+b+c)$$
Since $\beta(n)=\beta_n\neq 0$ for all $n\geq 1 \in \Z$, we must have 
$a+b,a+c,b+c\notin \Z_{<0}$.
Now, $q(0)=1\neq0$ thus $\beta(0)=\beta_0\neq0$ and thus 
$0\notin \{a+b,a+c,b+c\}$.

\item Continuous Hahn: 
We have $D(k)=w(k)=(k+1)(k+c)(k+d)$. Thus the condition $D(k)\neq 0 \,\forall k\in \Z_{\geq 0}$ implies $c,d\notin \Z_{\leq 0}$.
The condition $N(n+1,n)\neq 0$ and $N(n+2,n)\neq 0$ for all $n\geq 0\in \Z$ implies $b+c+d\notin \Z_{\leq 0}.$ The formula for $\beta$ in \cite[\S 9.4]{KLS} is
$$
\beta(u)=-\frac{(u+b+c+d-1)(u+c)(u+d)(u+1)(u+b+c)(u+b+d)}{(2u+b+c+d-1)(2u+b+c+d)^2(2u+b+c+d+1)}
$$
Thus $b+c,b+d\notin\Z_{<0}$. Suppose now that $0\in \{b+c,b+d\}$. Then $\beta(0)=0$, and thus $q(0)=0$. Thus $b+c+d=1$, and thus $1\in \{d,c\}$. This implies that $w$ has a double zero at $-1$, while $q$ has a simple zero at 0. Thus $\beta_0=\beta(0)=0$, contradicting $\beta(0)=0$.
Altogether we have $c,d,b+c+d,b+c,b+d\notin \Z_{\leq 0}$.

\item Meixner: we have $w(s)=d(s+1)(s+b)$ for some $d\in \C$. Since this polynomial must be non-zero, we have $d\neq 0$. Since $D(k)=w(k)\neq 0 \,\,\forall k\in \Z_{\geq 0}$, we must have $b\notin \Z_{\leq 0}$. It is left to show that $d$ has the form $1-1/c$ with $c\notin \{0,1\}$, {\it i.e.} that $d\neq 1$. The formula for $\beta$ for this family is
$$\beta(u)=\frac{(u+1)(u+b)(1-d)}{d^2},$$
and since $\beta\not \equiv 0$ we must have $d\neq 1$.
\item Charlier: we have $w(s)=d(s+1)$ for some $d\in \C$. Since this polynomial must be  non-zero, $d$ has the form $d=-1/a$ for some $a\neq 0$. \proofend
\end{enumerate}

\begin{remark}
By Proposition \ref{th:JacAlp}, we have
\begin{equation}
\alp_0=\alp(0)+f_{1}(-1)=\alp(0)+\frac{w(-1)}{p(-1)q(-1)}
\end{equation}
Since $w(-1)=0,$ this implies that $\alp_0=\alp(0)$, unless $p(-1)=0$ or $q(-1)=0$. This happens only in the following cases:

\begin{enumerate}[(a)]
\item Jacobi and Bessel families for $a=1$.
\item Wilson families with $a+b+c+d=2$.
\item Continuous Hahn family with $b+c+d=2$.
\item The $E_n$ family.
\end{enumerate}
\end{remark}



\section{Investigation of the two new families $E_{n}^{(c)},F_{n}^{(c)}$}\label{sec:new}
\subsection{Preliminaries on Lommel polynomials}
Let us give some preliminaries on Lommel polynomials from \cite[\S 9]{Wat} and \cite[\S 6.5]{Ism}. They are defined by the recursive relation
\begin{equation}\label{=Lom}
2z(n+c)h_{n}^{(c)}=h_{n+1}^{(c)}(z)+h^{(c)}_{n-1}(z),\quad\quad h_{-1}^{(c)}(z)\equiv 0, \,\, h^{(c)}_{0}\equiv 1
\end{equation}
Explicitly, we have\footnote{\DimaH{M. Ismail kindly informed us of a typo in \cite[(6.5.8)]{Ism}:  $z/2$ should be $2z$.}}
\begin{multline}\label{=LomExp}
h_{2n}^{(c)}(z)=(-1)^n\sum_{k=0}^n\binom{n+k}{2k}(n+c-k)_{(2k)}(-2z)^{2k}, \\ h_{2n+1}^{(c)}(z)=(-1)^n\sum_{k=0}^n\binom{n+k+1}{2k+1}(n+c-k)_{(2k+1)}(-2z)^{2k+1}
\end{multline}

To describe the discrete measure on the real line for which these polynomials are orthogonal we will need some notation. Let $J_v(z)$ denote the modified Bessel function \cite{Wat}:
\begin{equation}
J_v(z)=\sum_{n=0}^{\infty}\frac{(-1)^n(z/2)^{c+2n}}{\Gamma(n+\nu+1)n!}
\end{equation}
Let $\{j_{\nu,k}\}_{k=1}^{\infty}$ denote the increasing sequence of positive zeroes of $J_{\nu,k}$ on $\R$.
Denote $a_{\nu,k}:=j_{\nu,k}^{-1}$.

\begin{thm}[{\cite[(6.5.17)]{Ism}}]\label{thm:lomOrth}
For every $c\in \R_{>0}$ and every $n,m\in \Z_{\geq 0}$ we have
\begin{equation}
\sum_{k=1}^{\infty}a_{c-1,k}^2(h_{n,c}(a_{c-1,k})h_{m,c}(a_{c-1,k})+h_{n,c}(-a_{c-1,k})h_{m,c}(-a_{c-1,k}))=\frac{\delta_{m,n}}{2(n+c)}
\end{equation}
\end{thm}

\subsection{Discrete measure for which $\enc(z)$ and $\fnc(z)$ form orthogonal families}$\,$\\
Spelling out the definition of $E_n$ and $F_n$ as hypergeometric functions, we have
\begin{multline}\label{=EF_Exp}
E_{n}^{(c)}(z)=\frac{1}{2^{2n}c_{(2n)}}\sum_{k=0}^n\binom{n+k}{2k}(n+c-k)_{(2k)}2^{2k}z^{k}, \\ F_n^{(c)}(z)=\frac{1}{2^{2n}c_{(2n+1)}}\sum_{k=0}^n\binom{n+k+1}{2k+1}(n+c-k)_{(2k+1)}2^{2k}z^{k}
\end{multline}

Combining \eqref{=EF_Exp} and \eqref{=LomExp}
we have for any $n\in \Z_{\geq 0}$:
\begin{equation}\label{=EFLom}
\enc(-z^2) = \frac{(-1)^n}{2^{2n}c_{(2n)}} h_{2n}^{(c)}(z), \quad z\fnc(-z^2)=\frac{(-1)^{n}}{2^{2n+1}c_{(2n+1)}}h_{2n+1}^{(c)}(z).
\end{equation}
Theorem \ref{thm:lomOrth} and \eqref{=EFLom} imply the following corollary.
\begin{cor}
For every $c\in \R_{>0}$ and every $n,m\in \Z_{\geq 0}$ we have
\begin{equation}\label{=Em}
4c\sum_{k=1}^{\infty}a_{c-1,k}^2\enc(-a^2_{c-1,k})E_m^{(c)}(-a^2_{c-1,k})=\frac{c}{2^{4n}c^{2}_{(2n)}(2n+c)}\delta_{m,n}
\end{equation}
and
\begin{equation}\label{=Fm}
16c^2(c+1)\sum_{k=1}^{\infty}a_{c-1,k}^4\fnc(-a^2_{c-1,k})F_m^{(c)}(-a^2_{c-1,k})=\frac{c+1}{2^{4n}(c+1)^2_{(2n)}(2n+1+c)}\delta_{m,n}
\end{equation}
\end{cor}
\DimaG{The arguments of \cite[\S 6.5]{Ism} show that \eqref{=Fm} continues to hold for every \DimaK{non-zero} $c>-1$.}
\subsection{Differential equations}
Being hypergeometric functions of type $\,_4F_1$, the $\enc$ and $\fnc$ satisfy the following 4th order differential equations.
\begin{equation}
(D(D-1/2)-z(D-n)(D-n-c+1)(D+n+c)(D+n+1))\enc=0
\end{equation}
and
\begin{equation}
(D(D+1/2)-z(D-n)(D-n-c+1)(D+n+c+1)(D+n+2))\fnc=0,
\end{equation}
where $D=z\frac{\partial}{\partial_z}$ is the Euler operator.
\appendix
\section{Extensions to all fields of characteristic zero - proof of Theorem \ref{thm:allFields}}
\label{sec:allFields}

Let $F$ be a field of characteristic zero. Let $\cP=\{P_n\}$ be a quasi-orthogonal rational HG-family of  polynomials corresponding to a pair $(R,f)$.
Let $\Phi_k$ be the Newtonian basis defined by $R$, and let $c(n,k)$ be the family of coefficients of polynomials $P_n$ in the basis $\Phi_k$.
Define an $F$-algebra $A_F$ as in section \ref{sec:PolMod}. 
Let $M$ be the $A_F$-module defined by $c(n,k)$ and $v\in M$ be the element defined by $c(n,k)$. 
Our proof of Theorem \ref{thm:1dim} is completely algebraic, and holds over $F$ as well. Thus, $M$ is one-dimensional, and we have 
$Sv=fv$ and $Uv=hv$ for some $h\in K_F:=F(u,s)$. Then 
\begin{equation}\label{=F}
  \frac{Uf}{f}=\frac{Sh}{h} \text{ and }S^{-1}(f)^{-1} = h +\alp+R  + U^{-1}(\beta/h) \text{ for some }\alp,\beta\in F(u)
\end{equation}
Theorems \ref{thm:main} and \ref{thm:rat} are statements on all $f,h\in K_F$ satisfying \eqref{=F}. Every tuple $R,f,h,\alp,\beta$ is defined by finitely many elements $a_1,\dots,a_n$ of $F$, and the equations \eqref{=F} are algebraic conditions on these elements. In other words, these are statements about the subfield $L:=\bQ\langle a_1,\dots,a_n\rangle \subset F$ generated by $a_1,\dots,a_n$ over $\bQ$. 
Since $L$ has finite transcendence degree over $\bQ$, it can be embedded into the field $\bC$ of complex numbers.
Therefore, the tuple $(R,f,h,\alp,\beta)$ defines a tuple $(R_{\bC},f_{\bC},h_{\bC},\alp_{\bC},\beta_{\bC})$ that satisfies \eqref{=F} over $\bC$. 
This tuple in turn defines a meromorphic function $\Psi$ of the form $g\phi$, where $g\in \C(u,s)$ and $\phi$ is a meromorphic function of $\gamma$-type, such that $\Psi$ satisfies the conditions of Theorem \ref{thm:3termForm}, and therefore the analysis of \S\S \ref{sec:3t}-\ref{sec:PfRat} is valid for $f,h$. In particular, the analogue of Theorem \ref{thm:rat} holds for $(R,f)$. 

If $\cP$ is an HG family, then  
the analogue of Theorem \ref{thm:main} also holds for $(R,f)$ for the same reason. 
Proposition \ref{th:JacAlp} is a formal computation that holds over every field. Corollary \ref{cor:ab}\eqref{it:abqort}-\eqref{it:abchoice} and Theorem \ref{thm:Classmain} are also deduced from Theorem \ref{thm:main} and Proposition \ref{th:JacAlp} in the same way. \proofend

\section{Proof of technical lemmas}\label{sec:rat}

Recall that $\alp_n$ and $\beta_n$ denote the coefficients  in the 3-term recursion for $P_n$, {\it i.e.} for every $n>0$ we have 
\begin{equation}\label{=3tPP}
zP_n=P_{n+1}+\alp_nP_n+\beta_{n-1}{P_{n-1}} 
\end{equation}



Recall that we have relatively prime polynomials $N(u,s)$ and $D(u,s)$ such that
\begin{equation}\label{=hyp}
c(n, k + 1) D(n,k) = c(n, k) N (n, k)\quad \forall n,k\in \Z,
\end{equation} and  that we extend the domain of definition of the coefficients $c(n,k)$ to $\mathbb{Z}^2$ by setting $c(n,k) = 0$ outside of the range $0 \leq k \leq n$. We also have $c(n,n)=1$ for all $n\in \Z_{\geq 0}$. 

We will start with the following auxiliary lemmas.
\begin{lem}\label{lem:def}
Let $\{Q_n\}_{n=0}^{\infty}$ be a monic quasi-orthogonal family. Suppose that there exists $N\in \Z_{\geq0}$ such that $P_n=Q_n$ for all $n\geq N$. Then $P_n=Q_n$ for all $n$.
\end{lem}
\begin{proof}
 By the Gauss-Favard theorem (Theorem \ref{thm:Fav}), there exist sequences   $\gamma_n$, $\mu_n$  such that
\begin{equation}\label{=3tQ}
zQ_n=Q_{n+1}+\gamma_nQ_n+\mu_{n-1}{Q_{n-1}}.
\end{equation}

Let us prove by descending induction on $n$ that $P_n=Q_n$. The base is $n\in
\{N,N+1\}$. For the induction step assume $P_n=Q_n$ and $P_{n+1}=Q_{n+1}$. Then from (\ref{=3tPP},\ref{=3tQ}) we have 
\begin{equation}
zP_n-\alp_nP_n-\beta_{n-1}{P_{n-1}}=P_{n+1}=Q_{n+1}= zQ_n-\gamma_nQ_n-\mu_{n-1}{Q_{n-1}}
\end{equation}

Since $P_n=Q_n$  this implies
\begin{equation}
(\alp_n-\gamma_n)P_n+\beta_{n-1}{P_{n-1}}-\mu_{n-1}{Q_{n-1}}=0
\end{equation}
Since $P_n$ has degree $n$ and $P_{n-1},Q_{n-1}$ have degree $n-1$,  we have $\alp_n=\gamma_n$ and $P_{n-1}$ is proportional to $Q_{n-1}$. Since both polynomials are monic, they are equal. 
\end{proof}

\begin{lem}\label{lem:def2}
Let $\{Q^1_n\}_{n=0}^{\infty}$ and $\{Q^2_n\}_{n=0}^{\infty}$ be two monic quasi-orthogonal families. Suppose that for every $d$ there exists $N_n\in \Z_{\geq0}$ such that for every $n\geq N_d$ we have $\deg (Q^{1}_n-Q^{2}_n)< n-d$.  Then $Q^1_n=Q^2_n$ for all $n$.
\end{lem}

\begin{proof}
By the Gauss-Favard theorem (Theorem \ref{thm:Fav}), there exist sequences $\gamma^i_n$ and $\mu^i_{n-1}$  such that
\begin{equation}\label{=3tQQ}
zQ_n=Q_{n+1}+\gamma^i_nQ_n+\mu^i_{n-1}{Q_{n-1}}
\end{equation}
Denote the coefficients of the polynomials $Q^i_n$  by $a^i(n,k)$,  i.e. $Q^i_n=\sum z^ka^i(n,k)$. Considering the coefficient of $z^n$ in \eqref{=3tQQ} we have 
\begin{equation}
a^{i}(n,n-1)=a^{i}(n+1,n)+\gamma^i_n, \text{ thus }\gamma^i_n=a^{i}(n+1,n)-a^{i}(n,n-1)
\end{equation}
From the coefficient of $z^{n-1}$ we have
\begin{equation}
\mu^i_{n-1}=a^{i}(n,n-2)-a^{i}(n+1,n-1)-\gamma^i_na^{i}(n,n-1)
\end{equation}
Thus for $n\geq N_2$ we have $\gamma^1_n=\gamma^2_n$ and $\mu^1_{n-1}=\mu^2_{n-1}$.
Denote $Y_n:=Q^1_n-Q^2_n$.  
Let us show that 
\begin{equation}\label{=3tY}
zY_n=Y_{n+1}+\gamma^1_nY_n+\mu^1_{n-1}{Y_{n-1}} \text{ for all }n\geq N_2
\end{equation}
Indeed, for all $n\geq N_2$ we have
\begin{multline}
zY_n=z(Q^1_{n}-Q^2_{n})=Q^1_{n+1}+\gamma^1_nQ^1_n+\mu^1_{n-1}{Q^1_{n-1}}-Q^2_{n+1}-\gamma^2_nQ^2_n-\mu^2_{n-1}{Q^2_{n-1}}=\\
Q^1_{n+1}+\gamma^1_nQ^1_n+\mu^1_{n-1}{Q^1_{n-1}}-Q^2_{n+1}-\gamma^1_nQ^2_n-\mu^1_{n-1}{Q^1_{n-1}}=
Q^1_{n+1}-Q^2_{n+1}+\gamma^1_n(Q^1_n-Q^2_{n})+\mu^1_{n-1}({Q^1_{n-1}-Q^2_{n-1}})=\\
=Y_{n+1}+\gamma^1_nY_n+\mu^1_{n-1}{Y_{n-1}}
\end{multline}
Suppose by way of contradiction that $Y_{n_0}\not \equiv 0$ for some $n_0\geq 0$. Suppose first that $Y_{n_0}\not \equiv 0$ for some $n_0\geq N_2$.
Let $d_0:=n_{0}-\deg Y_{n_0}$. By \eqref{=3tY} we have $\deg Y_{n+1}=\deg Y_n+1$ for all $n\geq N_2$, and thus by induction $\deg Y_{n}=n-d_0$ for all $n\geq n_{0}$. However, by the condition of the lemma $\deg Y_n<n-d_0$ for $n\geq N_{d_0}$ reaching a contradiction. 

Thus we have $Y_{n}\equiv 0 $ for all $n>N_2$. Let us show by descending induction that $Y_{n}\equiv 0 $ for all $n$, reaching a contradiction. The base is $n\in
\{N_{2},N_{2}+1\}$. For the induction step assume $Y_n\equiv 0$ and $Y_{n+1}\equiv 0$. Then from (\ref{=3tQQ}) we have 
\begin{equation}
zQ^{1}_n-\gamma^1_nQ^1_n-\mu^1_{n-1}{Q^1_{n-1}}=Q^1_{n+1}=Q^2_{n+1}= zQ^2_n-\gamma^{2}_nQ^2_n-\mu^{2}_{n-1}{Q^2_{n-1}}
\end{equation}

Since $Q^{1}_n=Q^{2}_n$  this implies
\begin{equation}
(\gamma^1_n-\gamma^2_n)Q^{1}_n+\mu^1_{n-1}{Q^1_{n-1}}-\mu^2_{n-1}{Q^2_{n-1}}=0
\end{equation}
Since $Q^{1}_n$ has degree $n$ and $Q^{1}_{n-1},Q^{2}_{n-1}$ have degree $n-1$,  we have $\gamma^1_n=\gamma^2_n$ and $Q^1_{n-1}$ is proportional to $Q^2_{n-1}$. Since both polynomials are monic, they are equal. 
\end{proof}

\begin{proof}[Proof of Lemma \ref{lem:RDNdef}]

Let $\{Q^1_n\}_{n=0}^{\infty}$ and $\{Q^2_n\}_{n=0}^{\infty}$ be two monic quasi-orthogonal families with the same $R(s), D(u,s),N(u,s)$. 
By Lemma \ref{lem:def2}, it is enough to prove that for every $d\in \Z_{\geq 0}$ there exists $N_d\in \Z_{\geq0}$ such that for every $n\geq N_d$ we have $\deg (Q^{1}_n-Q^{2}_n)< n-d$.  

Fix $d\in \Z_{\geq 0}$. and for every $m\in \Z\cap [0,d]$ consider $N(u,u-m)$ as a polynomial in $u$. It has finitely many zeros, thus there exists $N_n\in \Z_{\geq0}$ such that for every $n\geq N_d$ we have $N(n,n-m)\neq 0$ for every  $m\in \Z\cap [0,d]$. For $j\in \{1,2\}$ let $c^{j}(n,k)$ denote the coefficients of $Q^1_n$ in the basis $\Phi_k$ defined by $R(k)$. 
Then by \eqref{=hyp0} we have $\forall n\geq N_d$ and $m\in \Z\cap [0,d]$: 
$$c^j(n,n-m)=\prod_{i=1}^m \frac{D(n,n-i)}{N(n,n-i)},$$
and thus $c^1(n,n-m)=c^2(n,n-m)$. Thus $\deg (Q^{1}_n-Q^{2}_n)< n-d$, and thus $Q^1_n=Q^2_n$ for all $n$.
\end{proof}

\begin{proof}[Proof of Lemma \ref{lem:fxyRat}] 
Substituting $n=k$ in \eqref{=hyp} we get $0=N(n,n)$. Thus the polynomial $N(u,u)$ vanishes identically, and thus $N$ is divisible by $s-u$. Substituting $k=-1$, we have $c(n,0)D(n,-1)=0$. We want to show that $D(n,-1)$ vanishes identically. Suppose the contrary. Then there exists a finite subset $Y\sub \Z_{\geq 0}$ such that $c(n,0)=0$ for all $n\in \Z_{\geq 0}\smallsetminus Y$ and $c(n,0)\neq 0$ for all $n\in Y$. The set $Y$ is non-empty since $c(0,0)=1$. Let $n=\max(Y) +1$. Then from the three-term recursion we get $0=c(n+1,0)+(\alp_n+R(0))c(n,0)+\beta_{n-1}c(n-1,0)=\beta_{n-1}c(n-1,0)$. Since $\beta_{n-1}\neq 0$, we get $c(n-1,0)=0$, contradicting $n-1\in Y$. Thus $D(u,-1)\equiv 0$, {\it i.e.} $D(u,s)$ is divisible by $s+1$. 
\end{proof}

\begin{proof}[Proof of Lemma \ref{lem:rat}]
Since $c(n,n)=1$, \eqref{=hyp} implies $c(n,n-1)N(n,n-1)=D(n,n-1)$. 
If $N(n,n-1)\equiv 0$ then $u-s-1$ divides $N$. Then $D(n,n-1)\equiv 0$ and thus $u-s-1$ divides also $D$. Since $N$ and $D$ are coprime, this is impossible, thus $N(n,n-1) \not \equiv 0$. Thus for $n$ outside a finite set we have 
\begin{equation}\label{=cnn_1}
c(n,n-1)=\frac{D(n,n-1)}{N(n,n-1)},
\end{equation} 
which is a well-defined rational function. 
Recall \eqref{=3tc}:
\begin{equation}\label{=3tcAp}
c(n,k-1)=c(n+1,k)+(\alp_n+R(k)) c(n,k)+\beta_{n-1}c(n-1,k)\quad \forall n,k\in\Z
\end{equation}
Substituting $k=n$ we obtain $\alp_n=c(n,n-1)-c(n+1,n)-R(n)$. By \eqref{=cnn_1} this is an \emph{almost rational function} on $\Z_{>0}$, {\it i.e.} coincides with a rational function outside a finite (possibly empty) set.

Substituting $k=n-1$ in \eqref{=3tcAp} we have 
\begin{equation}\label{=cn23}
c(n,n-2)=c(n+1,n-1)+(\alp_n+R(n-1))c(n,n-1)+\beta_{n-1}\quad \forall n\in\Z
\end{equation} 
By \eqref{=hyp} we have 
\begin{equation}\label{=cnn_2}
{N(n,n-2)}c(n,n-2)={D(n,n-2)}c(n,n-1)
\end{equation} 
If $N(n,n-2)\not \equiv 0$ then \eqref{=cnn_2} implies that $c(n,n-2)$ coincides with a rational function outside a finite set, and by \eqref{=cn23} so does $\beta_n$. It is left to consider the case  $N(n,n-2)\equiv 0$. Since $D,N$ are coprime, we have $D(n,n-2)\not\equiv 0$ and thus $c(n,n-1)\overset{a}{=}0$, {\it i.e.} vanishes outside a finite set. This implies $\alp_n\overset{a}{=}0$.
The set $X=\{l\in \Z_{> 0}\, \vert \, D(n,n-l)\equiv 0\}$ is finite and non-empty. Let $m=\max X$. 
Then $c(n,n-m+1)\cdot 0=c(n,n-m)N(n,n-m)$ and thus $c(n,n-m)\overset{a}{=}0$.
Denote $$Y:=\{l\in \Z_{\geq 0}\, \vert \, c(n,n-l)\overset{a}{=}0\}.$$ We just showed $m\in Y$. For any $l\in Y$ we have 
$$0\overset{a}{=}c(n,n-l)\overset{a}{=}c(n+1,n-l+1)+\beta_{n-1}c(n-1,n-l+1)\overset{a}{=}\beta_{n-1} c(n-1,n-l+1),$$
and thus $c(n-1,n-l+1)\overset{a}{=}0$ and $c(n,n-(l-2))\overset{a}{=}0$ and $l-2\in Y$ if $l\geq 2$. Since $0\notin Y$, this implies that $l$ is odd. By \eqref{=3tcAp} we have 
\begin{equation}
c(n,n-l-2)\overset{a}{=}c(n+1,n-l-1)+\beta_{n-1}c(n-1,n-l-1)\overset{a}{=}c(n+1,n-l-1)
\end{equation}
Thus $c(n,n-l-2)\overset{a}{=}d_l$ for some constant $d_{l}\in \C$. We have 
$$c(n,n-l-1)D(n,n-l-2)=d_lN(n,n-l-2).$$
If $d_l=0$ then $c(n,n-l-1)\overset{a}{=}0$, thus $l+1\in Y$, contradicting the fact that all elements of $Y$ are odd. 
 Thus $d_l\neq 0$, and thus $c(n,n-l-1)$ is an almost rational function (since $D(n,n-l-2)\not \equiv 0)$. 
By \eqref{=3tcAp} we have 
\begin{equation}
c(n,n-l-1)\overset{a}{=}c(n+1,n-l)+\beta_{n-1}c(n-1,n-l)
\end{equation}
Now, if $m=1$ then taking $l:=m$ we obtain that $\beta_{n-1}$ is an almost rational function. Otherwise, $m\geq 3$, and $m-2\in Y$. Thus $c(n,n-m+1)$ is an almost rational function, and so is its shift $c(n-1,n-m)$. Since $m-1\notin Y$, this almost rational function is not almost zero. Thus 
\begin{equation}
\beta_{n-1}\overset{a}{=}\frac{c(n+1,n-m)-c(n,n-m-1)}{c(n-1,n-m)}
\end{equation}
is an almost rational function.

Altogether, we showed that in all cases, both $\alp_n$ and $\beta_n$ coincide with  rational functions for all but finitely many $n$.
%
\end{proof}

\begin{proof}[Proof of Lemma \ref{lem:non0Sw}]
\DimaE{Let us first show that $w\neq 0$. Equivalently, we have to show that $c$ is not torsion. In other words, we have to show that the only polynomial $p\in \C[u,s]$ satisfying $p(n,k)c(n,k)=0$ for all $n,k\in \Z^2$ is $p=0$.
\DimaJ{
Let $p$ be such a polynomial, and suppose $p\neq 0$. For any $l$, consider the polynomial in $u$ given by $p(u,u-l)$. If it vanishes identically then $p$ is divisible by $u-s-l$. This can happen only for finitely many $l$.
For every $l$ for which this does not happen we have $c(n,n-l)\overset{a}{=}0$, {\it i.e.} $c(n,n-l)=0$ for all but finitely many $n$. Thus there exists a finite set $X\subset \Z_{\geq 0}$ such that for any $l\in \Z_{\geq 0}\smallsetminus X$ we have $c(n,n-l)\overset{a}{=}0$, but for $l\in X$ this does not hold.
Since $c(n,n)=1$, we have $0\in X$ and thus $X\neq \emptyset$. 
 Let $m=\max X$.  Then for any $l>m \in \Z$ we have $p(n,n-l)c(n,n-l)=0$, and thus $c(n,n-l)=0$ for all but finitely many $n\in \Z$. In particular, $c(n,n-m-2)\overset{a}{=}0$ and $c(n,n-m-1)\overset{a}{=}0$. By \eqref{=3tcAp} we have 
\begin{equation}
c(n,n-m-2)=c(n+1,n-m-1)+(\alp_n+R(n-m-1))c(n,n-m-1)+\beta_{n-1}c(n-1,n-m-1)
\end{equation}
and thus $\beta_{n-1}c(n-1,n-m-1)\overset{a}{=}0$. Since $\beta_n\neq 0$ for all $n\in \Z_{\geq 0}$ we get $c(n-1,n-m-1)\overset{a}{=}0$. Thus $c(n,n-m)\overset{a}{=}0$, contradicting $m\in X$.

}}

Let us now show that $Sw=fw$.
\DimaJ{By \eqref{=hyp} 
we have the equality 
$N(u,s)Sw=D(u,s)w$
 in $M$. This  implies $Sw=fw$.}
\end{proof}

\subsection{Proof of Lemma \ref{lem:fxy} on HG-families}\label{subsec:Jacrat}

\eqref{it:fxy}
We have to show that $D(s)$ has no zeros in $\Z_{\geq 0}$.

 Suppose the contrary: $D(k_0)=0$ for some $k_0\in \Z_{\geq 0}$.  By \eqref{=hyp} we have 
 \begin{equation}\label{=Lag0}
 c(n,k+1)D(k)=c(n,k)N(n,k)
 \end{equation}
This implies that $N(u,s)$ has no factors of the form $s-l_0$ for $l_0\in \Z_{\geq 0}$. Indeed, if such a factor existed then setting $n=l_0+1,k=l_0$ in \eqref{=Lag0} and using $c(l_0+1,l_0+1)=1$ we would obtain $D(l_0)=0$. Thus $D(s)$ and $N(u,s)$ would have a common factor, contradicting our assumptions.
 
 Setting $k=k_0$ in \eqref{=Lag0}  we have $c(n,k_0)=0$ for any $n$ with $N(n,k_0)\neq0$, {\it i.e.} $c(n,k_0)=0$ for all $n$ big enough.
By induction, using  \eqref{=Lag0}, this implies $c(n,i)=0$ for all $0\leq i\leq k_0$ and $n$ big enough.  In particular, the set $Y:=\{n \, \vert \, c(n,0)\neq0\}$ is finite. Let $n_0:=\max Y+1>0$. By \eqref{=3tcAp} we have
\begin{equation}\label{=3tc0}
0=c(n_0,-1)=c(n_0+1,0)+(\alp_{n_0}-R(0)) c(n_0,0)+\beta_{n_0-1}c(n_0-1,0)=\beta_{n_0-1}c(n_0-1,0)
\end{equation}
Since $n_0-1\in Y$, we have $c(n_0-1,0)\neq0$ and thus $\beta_{n_0-1}=0$, contradicting the condition that the family is quasi-orthogonal.

\eqref{it:xy!0} follows  from  \eqref{=Lag0}  by descending induction on $k$, using $c(n,n)=1$ and $D(k)\neq 0$.

\eqref{it:xyc} follows from \eqref{=Lag0}  by descending induction on $k\in [0,n]$, the base case being $c(n,n)=1$.
\proofend

\DimaJ{
\section{Examples of rational HG-families}\label{app:gauge}
All the examples of quasi-orthogonal \textbackslash\ HG-families that are not HG-families  that we found are built in the same way. We find two HG-families $P_n,Q_n$ that define the same module $M$, and have the same rational functions $\alp,\beta$, but have different $\alp_0$. Then we define a new family $R_n^{\mu}:=\mu P_n + (1-\mu)Q_n$. It lies in the same module and satisfies the same 3-term recursion for $n>0$, and thus is a rational HG family. 
In the three subsections below we do this for the Jacobi and continuous Hahn families with very special parameters,  and for the two new families $E_n^{(c)}$ and $F_n^{(c)}$ (that actually define the same module). In \S \ref{subsec:diff} we consider 
a bit more elaborate construction: we let $R_n:=P_{n+1}-Q_{n+1}$ (which has degree $\leq n$ if $P_n,Q_n$ are monic).
\DimaL{In \S \ref{subsec:ApTab} below we summarize the examples in a table.}
\subsection{Jacobi polynomials with $a=1,b=l+1/2$}\label{subsec:BJac}

Consider the Jacobi family for $a=1$. Then 
$$\alp=1/2, \text{ and }\beta=\frac{u^2-(b-1)^2}{4u^2-1}.$$

We see that for $b'= 2-b$  we have the same recurrence relation. However, the 2 families are different, since they have different starting condition: $p_1=x-b/2$. Thus, any linear combination of the two families is quasi-orthogonal (unless $b$ is an integer). Their average, that has the initial condition $p_1=x-1/2$, is called ``exceptional Jacobi polynomial' - see
\cite[Ch. 4, Remark 4.2.1]{Ism}. It is not always a rational HG-family, unless $b$ is a half-integer.
If $b$ is half-integer then any linear combination is a rational HG-family. 

Let $P_n$ denote the Jacobi family for the parameters $a=1, b=l+1/2$ where $l\in \Z_{>0}$, and $Q_n$ denote the Jacobi family for the parameters $a=1, b'=2-b=3/2-l$. Then $Q_{n,k}=\frac{r(k)}{r(n)}P_{n,k}$, where  
\begin{equation}
r(t)=(t-l+3/2)_{(2l-1)}
\end{equation}
Consider the family $R^{l,\lam}_n:=(\lam P_n+Q_n)/(\lam+1)$ for $\lam \neq -1$. 
Up to normalization, it is $g\cP$, where 
\begin{equation}
g(u,s)=\lam + \frac{r(s)}{r(u)}=\frac{r(s)+\lam r(u)}{r(u)}
\end{equation}
Consider first the case $l=1$. Then $r(t)=t+1/2$, and $g=\frac{\lam (u+1/2) +s +1/2}{u+1/2}$. Thus
\begin{equation}
R^{1,\lam}_n=\,_3F_2(-n,n+1,\lam n + \frac{\lam +3}{2}; 3/2,\lam n + \frac{\lam +1}{2} ;z)
\end{equation}
Thus this family satisfies a third order differential equation. For $\lambda=1$ this is the Jacobi family with $a=2$ and $b=3/2$. For $\lambda \notin \{-1,0,1\}$ this is a rational HG family, that is not an HG family.

From the formula for $\beta$, we see that this family orthogonal, for all real
$\lam \neq -1$. 

For $l>1$, the resulting $g$ does not seem to decompose to a product of linear factors. For $\lam =1$ and $l=2$, the numerator of $g$ is $(s+u+1)(4s^2 - 4su + 2s + 4u^2 + 2u - 3)$.

Actually, all the Jacobi families with $a=1, b=l+1/2$  lie in the same 1-dimensional module.
\subsubsection{The case $\lam=-1$}\label{subsec:diff}
The polynomial $Q_n-P_n$ will have degree less than $n$. Thus let us define \begin{equation}
R_n:=U(Q_n-P_n)
\end{equation}
For this family we have 
\begin{equation}
f=\frac{r(s+1)-r(u+1)}{r(s)-r(u+1)}\cdot \frac{(s-u-1)(s+u+2)}{(s+1)(s+l+1/2)}
\end{equation}
For $l=1$ we get that this is a Jacobi family with parameters $a=2,\,b=3/2$ in our notation. 
For $l=2$ we get a rational HG-family that is not an HG-family.

\subsection{The new families $E_n^{(c)}$, $F_n^{(c)}$}
Let $c,\lam\in \C$ with $c\notin \bZ_{\leq 0}$, and define 
\begin{multline}
g=(u+s+1)(s-u-c)-\lam(u+s+c+1)(s-u),\\ p(t):=t(t-c), \,\,q(t):=(t+1)(t+c+1), \,\,w(s):=(s+1)(s+3/2)
\end{multline}
and 
\begin{equation}\label{=frat}
f(u,s):=\frac{p(s-u)q(u+s)g(u,s+1)}{w(s)g(u,s)},\,\,h= \frac{w(u)q(u+s)g(u,u)g(u+1,s)}{q(2u)q(2u+1)p(s-u-1)g(u+1,u+1)g(u,s)}
\end{equation}
Let $\cP^{c,\lam}$ be the polynomial family defined by $f$ via the formula $c(n,k)=\prod_{i=k}^{n-1}f(n,i)^{-1}$.
Let 
\begin{equation}
\alp=-\frac{1}{2(2u+c)(2u+c+2)} \text{ and } \beta=\frac{1}{16(2u+c+1)(2u+c)^{2}(2u+c-1)}
\end{equation}
Then by a straightforward computation we get $c_{n,k+1}=f(n,k)c_{n,k},\,\,c_{n+1,k}=h(n,k)c_{n,k},$ and \begin{equation}
S^{-1}(f)^{-1} = h +\alp  + U^{-1}(\beta/h).
\end{equation}
Thus, $zP^{c,\lam}_n=P^{c,\lam}_{n+1}+\alp(n)P^{c,\lam}_n+\beta(n-1)P^{c,\lam}_{n-1}$ and thus the family is quasi-orthogonal.
Note that for $\lam=0$ we obtain the family $F^{(c)}_n$ and for $\lam=1$ we obtain the family $E_n^{(c+1)}$. The rational functions $\alp$ and $\beta$ do not depend on $\lam$ and are the same as for $F_n^{(c)}$ and $E_n^{(c+1)}$.
The difference is in $
\alp_0$.
 
For $\lam\notin \{0,1\}$ this is a rational HG-family, but not an HG-family, since the denominator of $f$ depends on $u$. For all $\lam$ this family is quasi-orthogonal, and for real $\lam$ and real $c>-1$ the family is orthogonal. 

\subsection{Continuous Hahn family with $b=1-l$, $c=l+1/2$, $d=1/2$}\label{subsec:BcH}

Let $l$ be a natural number. Similarly to \S \ref{subsec:BJac}, the  Continuous Hahn families with $b=1-l$, $c=l+1/2$, $d=1/2$ and with $b'=l$, $c'=3/2-l$, $d'=d=1/2$ lie in the same module, and have the same rational functions $\alp,\beta$. They also have $\beta_0=\beta(0)$, but different $\alp_0$. In this case we have 
$r(t)=(t+3/2-l)_{(2l-1)}$, and 
$$g(u,s)=\lam+\frac{r(s)}{r(u)}$$
In the special case $l=1$ we have 
$$R_n^\lam={}_4F_3(-n,n+1,\lam n +\frac{\lam+3}{2}, iz; \frac{3}{2},\frac{1}{2},\lam n+ \frac{\lam +1}{2};1)$$
For $\lam=1$ this is another continuous Hahn family, but for $\lam\notin\{-1,0,1\}$ it is a rational HG family that is not an HG family. For every real $\lam \neq -1$ this family is orthogonal.

For  $l=2$ the $g$ does not factor into linear factors, and probably also not for bigger $l$. 
For all integer $l> 1$ and all real $\lam \neq -1$, the family is not orthogonal, but is quasi-orthogonal.

In the case $\lam=-1$ we can define a different family, similar to \S \ref{subsec:diff}.

\subsection{Summarizing table}\label{subsec:ApTab}
The four families described in this subsection can be presented up to renormalization as $P_n=\sum_{k=0}^n g(n,k)c(n,k)\Phi_k$, where $g\in \C(u,s)$, and 
$c(n,k)$ and $\Phi_k$ are given by an HG family $\cQ$. The $g$ and the $\cQ$ are given by the following table.
{\small
\begin{equation}\label{tab:Ap}
\begin{tabular}{c c} 
\hline 
$\cQ$  & $g(u,s)$ \\  
\hline 

Jacobi with $a=1$, $b=l+1/2, \, l\in \Z_{>0}$& $(s-l+3/2)_{(2l-1)}+\lam(u-l+3/2)_{(2l-1)}, \, \lam \neq -1$   \\   \hline\\

Jacobi with $a=2$, $b=l+1/2, \, l\in \Z_{>0}$& $\frac{(s-l+3/2)_{(2l-1)}-(u-l+5/2)_{(2l-1)}}{s-u-1}$\\\\   \hline\\
$\,_4F_1(-n,-n-c,n+1,n+c+1; 3/2;z)$
 & $(u+s+1)(s-u-c)-\lam(u+s+c+1)(s-u)$\\
\hline 
Cont. Hahn with $b=1-l$, $c=l+1/2$, $d=1/2, \, l\in \Z_{>0}$& $(s-l+3/2)_{(2l-1)}+\lam(u-l+3/2)_{(2l-1)}, \, \lam \neq -1$\\\\   \hline\\

\end{tabular}
\end{equation}
}
}
\subsection*{Acknowledgements} It is a real pleasure to thank Ehud de Shalit and Don Zagier for fruitful discussions, and Tom Koornwinder and Mourad Ismail for their insightful and helpful remarks on a preliminary version of the paper. The research of J. Bernstein was partially supported by the ISF grant 1400/19. The research of S. Sahi was partially supported by NSF grants DMS-1939600 and 2001537, and Simons Foundation grant 509766. \Dima{The research of D. Gourevitch was partially supported by BSF grant 2019724 and ISF grant 1781/23. A significant part of the research was done during the visits of S. Sahi at the Weizmann Institute, the visit of D. Gourevitch at the IAS, and the visit of J. Bernstein to the MPIM, Bonn in 2023. We thank the three institutes for their hospitality, and the BSF for funding two of these visits.}


\end{document}